\documentclass[11pt]{amsart}
\usepackage{graphicx} % Required for inserting images
\usepackage{amsmath}
\usepackage{amsthm} % For proof environment
\usepackage{amssymb}
\usepackage{verbatim} % For multiline commenting
\usepackage{color}
\usepackage[colorlinks=true, linkcolor=purple, urlcolor=purple, citecolor=purple]{hyperref}
\usepackage{enumitem}
\usepackage{tikz}
\usepackage{tikz-cd}
\usetikzlibrary{positioning}
\usepackage{forloop}

\title{A Fra\"{i}ss\'{e} theory for partial orders of a fixed finite dimension}

\author[I. B. Smythe]{Iian B. Smythe}
\address{Department of Mathematics and Statistics, University of Winnipeg, 515 Portage Avenue, Winnipeg, Manitoba, Canada, R3B 2E9}
\email{i.smythe@uwinnipeg.ca}
\urladdr{www.iiansmythe.com}

\author[M. Threz]{Mithuna Threz}
\address{Department of Mathematics, University of Manitoba, Winnipeg, Manitoba, Canada, R3T 2N2}

\author[M. Wiebe]{Max Wiebe}
\address{Department of Mathematics, University of Manitoba, Winnipeg, Manitoba, Canada, R3T 2N2}

\date{\today}
 %06A07 Combinatorics of partially ordered sets, 03C15 Model theory of denumerable and separable structures, 05D10 Ramsey theory, 37B05 Dynamical systems involving transformations and group actions with special properties
\thanks{This research was conducted as part of two undergraduate summer research projects under the supervision of the first author and supported by his NSERC Discovery Grant (RGPIN-2023-03343). The second author was also supported by two Undergraduate Summer Research Awards -- International from the University of Winnipeg. The authors gratefully acknowledge these organizations and their support for undergraduate research.}

\newtheorem{theorem}{Theorem}[section]
\newtheorem{corollary}[theorem]{Corollary}
\newtheorem{lemma}[theorem]{Lemma}
\newtheorem{proposition}[theorem]{Proposition}

\theoremstyle{definition}
\newtheorem{definition}[theorem]{Definition}
\newtheorem{example}[theorem]{Example}

\begin{document}

\begin{abstract}
    For each $n\geq 2$, we show that the class of all finite $n$-dimensional partial orders, when expanded with $n$ linear orders which realize the partial order, forms a Fra\"iss\'e class and identify its Fra\"iss\'e limit $(D_n,<,<_1,\ldots,<_n)$. We give a finite axiomatization of this limit which specifies it uniquely up to isomorphism among countable structures, show that its class of finite substructures satisfies the Ramsey property, and conclude, by the Kechris--Pestov--Todor\v{c}evi\'{c} correspondence, that the automorphism group of the limit is extremely amenable. We then identify the universal minimal flow of the automorphism group of the reduct $(D_n,<)$. Similar results are established for the $n$-dimensional rational grid $(\mathbb{Q}^n,<)$ and its expansion by the coordinate orders.
\end{abstract}

\maketitle

\section*{Introduction}

In his 1895 \textit{Beitr\"age} \cite{Cantor1895}, Cantor proved that the set $\mathbb{Q}$ of rational numbers, with its usual ordering, is the unique countable dense linear order without endpoints, up to isomorphism. The most well-known proof of this fact (e.g.,~in \cite{MR1697766}) is by a ``back-and-forth'' argument:\footnote{While the ``back-and-forth'' argument is often credited to Cantor, his original proof uses only the ``forth'' direction; see \cite{MR1281169} for the history of this proof.} Given a countable dense linear order without endpoints $(L,<)$, we enumerate both $L$ and $\mathbb{Q}$, and construct an isomorphism from $(L,<)$ to $(\mathbb{Q},<)$ by building an increasing chain of isomorphisms between finite substructures. To extend the domain of our isomorphism to the next enumerated point of $L$, we find a corresponding point in $\mathbb{Q}$ (going ``forth''), and to extend the range to the next enumerated point of $\mathbb{Q}$, we find a corresponding point in $L$ (going ``back''). Density and the lack of endpoints ensure that these extensions are always possible.

Fra\"iss\'e \cite{MR0069239} gave a new interpretation of Cantor's result by showing that $(\mathbb{Q},<)$ can be viewed as the ``limit'' of its class of finite substructures, namely the finite linear orderings. More generally, Fra\"iss\'e showed that whenever a class of finite structures satisfies certain properties, all of its members embed into a unique countable ultrahomogeneous structure, up to isomorphism, now called its \emph{Fra\"iss\'e limit}.

This line of research connects to topological dynamics via the study of automorphism groups of countable structures. The most remarkable such connection, established by Kechris, Pestov, and Todor\v{c}evi\'{c} in \cite{MR2140630}, says that a certain strong fixed-point property, \emph{extreme amenability}, of the automorphism group of a Fra\"iss\'e limit of a class of finite structures endowed with a linear ordering, is equivalent to a finite combinatorial \emph{Ramsey property} for that class. For example, the classical form of Ramsey's Theorem \cite{MR1576401} can be viewed as saying that the Ramsey property holds for the class of finite linear orders and thus, via this correspondence, the automorphism group of $(\mathbb{Q},<)$ is extremely amenable, a result original proved by Pestov \cite{MR1608494}. In cases where the underlying class of finite structures does not have the Ramsey property, but can be expanded to one which does, techniques from \cite{MR2140630} can be used to identify the largest irreducible action, or \emph{universal minimal flow}, of the automorphism group of its Fra\"iss\'e limit.

The goal of the present article is to establish generalizations of the above results to partial orders of an arbitrary fixed finite dimension.

The notion of \emph{dimension} for a partial order $(P,<)$ was defined by Dushnik and Miller in \cite{MR0004862} to be the minimum number of linear orders on $P$ whose intersection is $<$. The prototypical examples of $n$-dimensional partial orders are $n$-fold direct products of linear orders. Consequently, every countable $n$-dimensional partial order embeds into the rational grid $(\mathbb{Q}^n,<)$. We will focus particular attention on certain topologically dense subsets $D_n$ of $\mathbb{Q}^n$, isolated by Manaster and Remmel in \cite{MR0641492}.

The \textit{prima facie} difficulty in attempting to develop a Fra\"iss\'e theory for $n$-dimensional partial orders is that, when $n\geq 2$, they do not form a Fra\"iss\'e class (Proposition \ref{prop:PO_n_not_AP}), nor are the natural ``limit'' objects $(\mathbb{Q}^n,<)$ or $(D_n,<)$ ultrahomogeneous (Proposition \ref{prop:D_n_not_hom}). The solution, as we will see, is to expand these structures by the $n$ linear orders (or ``realizers'') in the case of $D_n$, or $n$ weak linear orders (or ``weak realizers'') in the case of $\mathbb{Q}^n$, which intersect to give the underlying partial order.

This article is arranged as follows: We begin by reviewing partial orders and dimension in Section \ref{sec:posets}, and then Fra\"iss\'e theory in Section \ref{sec:Fraisse}. We show, in Section \ref{sec:limit}, that the class of finite $n$-dimensional partial orders with realizers is a Fra\"iss\'e class with Fra\"iss\'e limit $(D_n,<,<_1,\ldots,<_n)$, where $<_1,\ldots,<_n$ are the $n$ coordinate orderings inherited from $\mathbb{Q}^n$ (Theorem \ref{thm:D_n_Fraisse}). In Section \ref{sec:axioms}, we give a finite axiomatization of the first-order theory of this Fra\"iss\'e limit which specifies it uniquely up to isomorphism among countable structures (Theorem \ref{thm:axiomatization}). In Section \ref{sec:Ramsey}, we show, as a consequence of the Product Ramsey Theorem of Graham, Rothschild, and Spencer \cite{MR3288500}, that the class of finite $n$-dimensional partial orders with realizers satisfies the Ramsey property (Theorem \ref{thm:PO_n_RP}) and conclude that the automorphism group of $(D_n,<,<_1\ldots,<_n)$ is extremely amenable (Corollary \ref{cor:EA}). In Section \ref{sec:uni_min_flow}, we identify the universal minimal flow of the automorphism group of $(D_n,<)$ as its natural action on the (finite) set of realizers for the underlying partial order (Theorem \ref{thm:UMF}). Finally, in Section \ref{sec:Qn}, we carry out a similar analysis for $(\mathbb{Q}^n,\leq)$ and its expansion by the (weak) coordinate orders $\leq_1,\ldots,\leq_n$.

Before proceeding further, we note that some of our results, namely the fact that the finite $n$-dimensional partial orders with realizers form a Fra\"iss\'e class and have the Ramsey property, are equivalent to results first proved by Bodirsky \cite{MR3210656} and Soki\'c \cite{MR3047085}, a fact of which we only became aware after establishing them on our own. However, in neither \cite{MR3210656} nor \cite{MR3047085} is the Fra\"iss\'e limit of this class explicitly described, and our proof of the Ramsey property appears to be both simpler and more direct. Our other results, in particular those concerning universal minimal flows, are, to the best our knowledge, completely new.

\section{Partial orders and dimension}\label{sec:posets}

By a \emph{partial order} we will mean both an irreflexive and transitive relation $<$ on a set $P$, as well as the structure $(P,<)$. Elements $a,b\in P$ are \emph{comparable} if $a< b$, $a=b$, or $b< a$; otherwise, they are \emph{incomparable}. A subset of $P$ consisting of pairwise incomparable elements is called an \emph{antichain}. A \emph{linear order} is a partial order in which every pair of elements is comparable. We will write $\leq$ for the non-strict order corresponding to $<$, where $a\leq b$ if and only if $a=b$ or $a<b$. We caution readers that we will often use the same symbols (e.g.,~$<$ or $<_i$) for different orders on different sets at the same time, however their individual meanings should be clear from context.

Given partial orders $(P,<)$ and $(Q,<)$, an \emph{embedding} of $(P,<)$ into $(Q,<)$ is an injective function $f:P\to Q$ such that
\[
    a<b \iff f(a)<f(b)
\]
for all $a,b\in P$. A bijective embedding from $(P,<)$ onto $(Q,<)$ is an \emph{isomorphism} and if such an isomorphism exists, we say that $(P,<)$ and $(Q,<)$ are \emph{isomorphic}. 

There are several different ways to form the ``product'' of partial orders $(P,<)$ and $(Q,<)$. First, there is the \emph{direct product} $(P\times Q,<)$, endowed with the \emph{product order}:
\[
    (a,b)<(c,d) \iff \text{$a\leq c$, $b\leq d$, and $(a,b)\neq (c,d)$},
\]
for all $(a,b),(c,d)\in P\times Q$. This definition naturally generalizes to products of more than two partial orders and is linear only in trivial cases.

Second, we have the \emph{lexicographic product} $(P\times Q,<_{\mathrm{lex}})$, endowed with the \emph{lexicographic order}:
\[
    (a,b)<_{\mathrm{lex}}(c,d) \iff \begin{cases}a<c, \text{ or}\\a=c \text{ and } b<d,\end{cases}
\]
for all $(a,b),(c,d)\in P\times Q$. It will be convenient to refer to this as the \emph{1st lexicographic order} on $P\times Q$, and write it as $<_{\mathrm{lex}_1}$. We define the \emph{2nd lexicographic order} (or \emph{colexicographic order}) $<_{\mathrm{lex}_2}$ on $P\times Q$ similarly, by considering the 2nd coordinate first. 

Given partial orders $(P_1,<),\ldots,(P_n,<)$, we will use boldface $\mathbf{a}$, $\mathbf{b}$, $\mathbf{c}$, to denote elements of the product $P_1\times\cdots\times P_n$, and write $a_i$, $b_i$, $c_i$, respectively, for their $i$th coordinates in $P_i$. For each $i\leq n$, we define the \emph{$i$th lexicographic order} $<_{\mathrm{lex}_i}$ on $P_1\times\cdots\times P_n$ by:
\[
    \mathbf{a}<_{\mathrm{lex}_i}\mathbf{b} \iff 
    \begin{cases}
        a_{\sigma_i(1)}<b_{\sigma_i(1)},\text{ or}\\
        a_{\sigma_i(1)}=b_{\sigma_i(1)} \text{ and } a_{\sigma_i(2)}<b_{\sigma_i(2)},\text{ or}\\
        \vdots\\
        \forall j<n (a_{\sigma_i(j)}=b_{\sigma_i(j)}) \text{ and } a_{\sigma_i(n)}<b_{\sigma_i(n)},
    \end{cases}
\]
where $\sigma_i$ is the cyclic permutation of $\{1,\ldots,n\}$ given by $\sigma_i(1)=i$, $\sigma_i(2)=i+1\bmod(n+1)$, $\sigma_i(3)=i+2\bmod(n+1)$, and so on.\footnote{The particular choice of permutation here is somewhat arbitrary, so long as $\sigma_i(1)=i$.} 
In particular,
\[
    \mathbf{a}<_{\mathrm{lex}_i}\mathbf{b} \implies a_i\leq b_i.
\]
Note that a lexicographic product of linear orders is always linear.

Given a partial order $(P,<)$, linear orders $<_1,\ldots,<_n$ on $P$ \emph{realize} $<$ if
\[
    <\;=\;<_1\cap\cdots\cap <_n,
\]
that is,
\[
    a<b \iff   a<_i b\text{ for all $i\leq n$,}
\]
for all $a,b\in P$. The \emph{dimension} of $(P,<)$, written $\dim(P,<)$, is the minimum number of linear orders on $P$ required to realize $<$. In particular, linear orders are exactly the partial orders of dimension $1$. The following fact, proved by a standard application of Zorn's Lemma, implies that (possibly infinite cardinal) dimension is well-defined for any partial order:

\begin{lemma}[Szpilrajn \cite{zbMATH02566488}]\label{lem:szpilrajn}
    Every partial order $<$ on $P$ can be extended to a linear order on $P$. Moreover, if $a,b\in P$ are incomparable, then there is a linear extension $\prec$ of $<$ on $P$ such that $a\prec b$.
\end{lemma}

We will focus on those partial orders having finite dimension. It is easy to see that every finite partial order $(P,<)$ has finite dimension, and in fact, provided $|P|\geq 4$,
\begin{equation}\label{eqn:hiraguchi}
    \dim(P,<)\leq \left\lfloor{\frac{|P|}{2}}\right\rfloor,
\end{equation}
a sharp bound due to Hiraguchi \cite{MR0070681}. 

We abuse terminology slightly and call partial orders of dimension $\leq\!n$ simply \emph{$n$-dimensional}, and denote by $\mathcal{PO}_n$ the \emph{class of all finite $n$-dimensional partial orders}.

\begin{example}\label{ex:crown}
    Komm \cite{MR25541} showed that for any $n\in\mathbb{N}$, the dimension of $(\mathcal{P}(\{1,\ldots,n\}),\subset)$, where $\subset$ is the proper subset relation on the power set of $\{1,\ldots,n\}$, is exactly $n$. In fact, if we just consider singletons and their complements in $\{1,\ldots,n\}$, this already has dimension $n$, as observed in \cite{MR0070681}. We can view the latter as consisting of two disjoint antichains, $a_1,\ldots,a_n$ and $b_1,\ldots,b_n$, with $b_i<a_j$ if and only if $i\neq j$; we call this partial order the \emph{$n$-crown} and denote it by $(C_n,<)$. By Hiraguchi's inequality (\ref{eqn:hiraguchi}), $(C_n,<)$ is a minimal example of a partial order having dimension $n$.
    \begin{figure}[ht]
    \scalebox{0.75}{
    \begin{tikzpicture}
        %style of the nodes used
        \tikzstyle{node_style_a}=[shape=circle,very thin, draw=black]
        \tikzstyle{node_style_b}=[shape=circle,very thin, draw=black, fill=gray]
        \node[node_style_a] (a1) at (1,1) {$a_1$};
        \node[node_style_a,right= of a1] (a2) {$a_2$};
        \node[node_style_a,right=of a2] (a3) {$a_3$};
        \node [node_style_b,above=of a1] (b1) {$b_1$};
        \node [node_style_b,right=of b1] (b2) {$b_2$};
        \node [node_style_b,right=of b2] (b3) {$b_3$};
        \draw (a1) -- (b2) (a1) -- (b3); %(a1) -- (bn);
        \draw (a2) -- (b1) (a2) -- (b3); %(a2) -- (bn);
        \draw (a3) -- (b1) (a3) -- (b2); %(a3) -- (bn);
    \end{tikzpicture}}
    \caption{A Hasse diagram for $(C_3,<)$.}
    \end{figure}
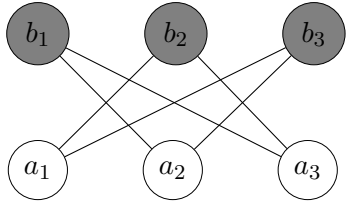
\end{example}

We note that, for any $n\geq2$, there is no finite list of ``forbidden'' partial orders with respect to the property of being $n$-dimensional. In fact, the class of $n$-dimensional partial orders is not even finitely axiomatizable in the first-order language of partial orders \cite{MR0300944}.

If $(L_1,<),\ldots,(L_n,<)$ are linear orders, $<$ their product order, and $<_{\mathrm{lex}_1},\ldots,<_{\mathrm{lex}_n}$ the lexicographic orders on $L_1\times\cdots\times L_n$ defined above, then $\mathbf{a}<\mathbf{b}$ if and only if $\mathbf{a}<_{\mathrm{lex}_i}\mathbf{b}$ for all $i\leq n$, and so $<_{\mathrm{lex}_1},\ldots,<_{\mathrm{lex}_n}$ realize the product order. This shows that $(L_1\times\cdots\times L_n,<)$, and thus any suborder, is $n$-dimensional. Conversely, if $(P,<)$ is $n$-dimensional, say with realizers $<_1,\ldots,<_n$, then the mapping $p\mapsto(p,\ldots,p)$ is an embedding of $(P,<)$ into the product of $(P,<_1),\ldots,(P,<_n)$ with the product order. This discussion proves the following alternate characterization of dimension:

\begin{theorem}[Theorem 10.4.2 in \cite{MR0150753}]\label{thm:Ore}
    A partial order $(P,<)$ is $n$-dimensional if and only if it embeds into an $n$-fold product of linear orders.\qed
\end{theorem}

It follows from Theorem \ref{thm:Ore} that every countable $n$-dimensional partial order embeds into the product $(\mathbb{Q}^n,<)$. Given what we know about $(\mathbb{Q},<)$, this may suggest that $(\mathbb{Q}^n,<)$ is the natural ``limit'' of the class of all finite $n$-dimensional partial orders. However, the presence of points in $(\mathbb{Q}^n,<)$ which are \emph{colinear}, i.e.,~share a common coordinate, will prevent it from being the right limit object in our initial analysis (see Proposition \ref{prop:Qn_not_hom} below). Instead, we will consider subsets of $\mathbb{Q}^n$ which are topologically dense, with respect to the usual product topology on $\mathbb{Q}^n$, but such that no two distinct points are colinear. Manaster and Remmel isolated these sets in \cite{MR0641492} (see also Theorem 6A in \cite{MR560230}) and showed that if $D_n$ was such a subset of $\mathbb{Q}^n$, then $(D_n,<)$ is universal for countable $n$-dimensional partial orders and is the model companion of the theory of $n$-dimensional partial orderings. These sets are not explicitly constructed in \cite{MR0641492}, so we do so here for reference:

\begin{proposition}\label{prop:D_n_exists}
    For each $n\geq 2$, there exist sets $D_n\subseteq\mathbb{Q}^n$ which are topologically dense and such that no two distinct points are colinear.
\end{proposition}

\begin{proof}
   Let $\{B_m:m\in\mathbb{N}\}$ enumerate all open balls with rational radii centered on points in $\mathbb{Q}^n$. We will construct $D_n$ recursively. 

   Begin by choosing $\mathbf{d}_0\in B_0$ arbitrarily. Observe that if we remove the $n$ hyperplanes consisting of all points in $\mathbb{Q}^n$ which share a coordinate with $\mathbf{d}_0$ from the open ball $B_1$, then there are still infinitely many points remaining in $B_1$; choose one of them as $\mathbf{d}_1$. Continue in this fashion. Having chosen $\mathbf{d}_0,\ldots,\mathbf{d}_m$, remove from $B_{m+1}$ the finitely hyperplanes sharing coordinates with any of $\mathbf{d}_0,\ldots,\mathbf{d}_m$, and choose a remaining point in $B_{m+1}$ as $\mathbf{d}_{m+1}$. Then, $D_n=\{\mathbf{d}_m:m\in\mathbb{N}\}$ clearly has the properties desired.
\end{proof}

This construction shows that there are, in fact, $2^{\aleph_0}$-many such subsets of $\mathbb{Q}^n$, as one can make many choices at each step. It will be a consequence of Theorem \ref{thm:axiomatization} below that it does not matter which one we choose, so for each $n\geq 2$, we will fix one such set $D_n$ throughout and write $\mathbf{D}_n$ for $(D_n,<)$.

If $<$ is the product order and $<_{\mathrm{lex}_1},\ldots,<_{\mathrm{lex}_n}$ the $n$ lexicographic orders on $D_n$ inherited from $\mathbb{Q}^n$, then as before, the latter realize the former. Moreover, since no two distinct points of $D_n$ are colinear, we have that
\[
    \mathbf{a}<_{\mathrm{lex}_i}\mathbf{b} \iff a_i<b_i,
\]
and so $<_{\mathrm{lex}_i}$ is just the $i$th coordinate order on $D_n$. For this reason, we will write $<_i$ for $<_{\mathrm{lex}_i}$ when referring to $D_n$, and let $\mathbf{D}_{n,<_1,\ldots,<_n}$ denote the resulting structure $(D_n,<,<_1,\ldots,<_n)$.

In \cite{MR560230}, Manaster and Rosenstein pointed out that $2$-dimensional partial orders are far more tractable, from a computability standpoint, when endowed with the two linear orders which realize the underlying partial order. We make a similar move here. That is, we will consider \emph{$n$-dimensional partial orders with realizers} as structures of the form $(P,<,<_1,\ldots,<_n)$, expanded by $n$ linear orders $<_1,\ldots,<_n$ which realize $<$, and denote by $\mathcal{PO}_{n,<_1,\ldots,<_n}$ the \emph{class of all finite $n$-dimensional partial orders with realizers}.

\section{Fra\"{i}ss\'{e} theory}\label{sec:Fraisse}

In this section, we review the basic notions from Fra\"iss\'e theory; a standard reference for this material is Chapter 7 of \cite{MR1221741}. We will confine our attention to structures whose \emph{signature} consists of finitely many binary relation symbols, say $R_0,R_1,\ldots,R_n$. A \emph{structure} in this signature is then an object of the form 
\[
    \mathbf{A}=(A,R_0^\mathbf{A},R_1^\mathbf{A},\ldots,R_n^{\mathbf{A}}),
\]
where $A$ is a non-empty set and each $R_i^\mathbf{A}$ is a binary relation on $A$. We will drop the superscript on $R_i^{\mathbf{A}}$, and simply write $R_i$, when $\mathbf{A}$ is understood.

If $\mathbf{A}$ and $\mathbf{B}$ are structures in the same signature, then an \emph{embedding} from $\mathbf{A}$ to $\mathbf{B}$ is an injective function $f:A\to B$ such that
\[
    a R_i^\mathbf{A} b \iff f(a) R_i^\mathbf{B} f(b)
\]
for all $a,b\in A$ and $i\leq n$. A bijective embedding is an \emph{isomorphism}. We write $\mathbf{A}\preceq\mathbf{B}$ if there is an embedding from $\mathbf{A}$ to $\mathbf{B}$, and $\mathbf{A}\cong\mathbf{B}$ if there is an isomorphism. If $A\subseteq B$ and $R_i^\mathbf{A}=R_i^\mathbf{B}\cap A^2$ for each $i\leq n$ (i.e., the inclusion map $A\hookrightarrow B$ is an embedding), then we call $\mathbf{A}$ a \emph{substructure} of $\mathbf{B}$ and write $\mathbf{A}\subseteq\mathbf{B}$.  An isomorphism from $\mathbf{A}$ to itself is called an \emph{automorphism} and we denote by $\mathrm{Aut}(\mathbf{A})$ the group of all automorphisms of $\mathbf{A}$. We say that $\mathbf{A}$ is \emph{ultrahomogeneous} if every isomorphism between finite substructures of $\mathbf{A}$ extends to an automorphism of $\mathbf{A}$.

Given a class $\mathcal{K}$ of finite structures in a fixed signature, we consider the following three properties which may or may not hold of $\mathcal{K}$:
\begin{itemize}
    \item The \emph{hereditary property (HP)}: If $\mathbf{A}\in\mathcal{K}$ and $\mathbf{B}\preceq\mathbf{A}$, then $\mathbf{B}\in\mathcal{K}$.
    \item The \emph{joint embedding property (JEP)}: If $\mathbf{A},\mathbf{B}\in\mathcal{K}$, then there is a $\mathbf{C}\in\mathcal{K}$ such that $\mathbf{A},\mathbf{B}\preceq\mathbf{C}$.
    \item The \emph{amalgamation property (AP)}: If $\mathbf{A},\mathbf{B},\mathbf{C}\in\mathcal{K}$ and $e:A\to B$ and $f:A\to C$ are embeddings, then there is a $\mathbf{D}\in\mathcal{K}$ and embeddings $g:B\to D$ and $h:C\to D$ such that $g\circ e=h\circ f$.
\end{itemize}

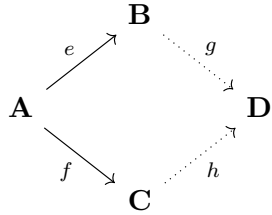
\begin{figure}[ht]
    \begin{tikzcd}
        & \mathbf{B} \arrow[dr,dotted,"g"] &\\
        \mathbf{A} \arrow[ur,"e"] \arrow[dr,swap,"f"] & & \mathbf{D} \\
        & \mathbf{C} \arrow[ur,dotted,swap,"h"] &
    \end{tikzcd}
    \caption{The amalgamation property (AP).}
    \label{fig:AP}
\end{figure}

If $\mathbf{A}$ is an infinite structure, then its \emph{age} is the class $\mathrm{Age}(\mathbf{A})$ of all finite structures in the same signature that embed into $\mathbf{A}$. It is easy to see that $\mathrm{Age}(\mathbf{A})$ always satisfies the HP and JEP. The following fundamental results are due to Fra\"iss\'e:

\begin{theorem}[Fra\"iss\'e \cite{MR0069239}, Theorem 7.1.2 in \cite{MR1221741}]\label{thm:Fraisse1}
    Let $\mathcal{K}$ be a non-empty class of finite structures in a fixed signature. If $\mathcal{K}$ satisfies the HP, JEP, and AP, then there is a unique countable ultrahomogeneous structure $\mathbf{M}$, up to isomorphism, such that $\mathrm{Age}(\mathbf{M})=\mathcal{K}$.
\end{theorem}

We call such a $\mathcal{K}$ a \emph{Fra\"iss\'e class} and $\mathbf{M}$ its \emph{Fra\"iss\'e limit}. Conversely:

\begin{theorem}[Fra\"iss\'e \cite{MR0069239}, Theorem 7.1.7 in \cite{MR1221741}]\label{thm:Fraisse2}
    If $\mathbf{M}$ is a countable ultrahomogeneous structure, then $\mathrm{Age}(\mathbf{M})$ is a Fra\"iss\'e class.
\end{theorem}

The standard back-and-forth argument used to prove Cantor's characterization of $(\mathbb{Q},<)$ can be used to show that $(\mathbb{Q},<)$ is ultrahomogeneous, and consequently, is the Fra\"iss\'e limit of the class $\mathcal{LO}$ of all finite linear orders. It can also be shown that the class $\mathcal{PO}$ of all finite partial orders forms a Fra\"iss\'e class, and thus has a Fra\"iss\'e limit, known as the \emph{generic partial order}; see \cite{MR2800979} for this and other examples. However, for any $n\geq 2$, the class $\mathcal{PO}_n$ of all finite $n$-dimensional partial orders is not a Fra\"iss\'e class:

\begin{proposition}\label{prop:PO_n_not_AP}
    For any $n\geq 2$, $\mathcal{PO}_n$ does not satisfy the amalgamation property.
\end{proposition}

\begin{proof}
    Let $(C_{n+1},<)$ be the $(n+1)$-crown as in Example \ref{ex:crown}, where
    \[
        C_{n+1}=\{a_1,a_2,\ldots,a_{n+1}\}\cup\{b_1,b_2,\ldots,b_{n+1}\}
    \]
    and $a_i<b_j$ if and only if $i\neq j$. Let 
    \begin{align*}
        A&=\{a_1,a_2,\ldots,a_{n+1}\},\\
        B&=A\cup\{b_2,\ldots,b_{n+1}\},\\
        C&=A\cup\{b_1\},
    \end{align*}
    and consider $(A,<)$, $(B,<)$, and $(C,<)$ as suborders of $(C_{n+1},<)$. Let $e$ and $f$ be the natural embeddings of $(A,<)$ into $(B,<)$ and $(C,<)$ respectively.

    \begin{figure}[ht]
    \scalebox{.50}{
    \begin{tikzpicture}
        \tikzstyle{node_style_a}=[shape=circle,very thin, draw=black]
        \tikzstyle{node_style_b}=[shape=circle,very thin, draw=black, fill=gray]

        \node[node_style_a] (a1) at (1,1) {$a_1$};
        \node[node_style_a,right=of a1] (a2) {$a_2$};
        \node[node_style_a,right=of a2] (a3) {$a_3$};
        \node[below=0.25cm of a2] {\huge{$A$}};

        \node[node_style_a, above right=4cm and 4cm of a3] (aa1) {$a_1$};
        \node[node_style_a,right=of aa1] (aa2) {$a_2$};
        \node[node_style_a,right=of aa2] (aa3) {$a_3$};
        \node[node_style_b,above=of aa2] (bb2) {$b_2$};
        \node[node_style_b,above=of aa3] (bb3) {$b_3$};
        \draw (aa1) -- (bb2) (aa1) -- (bb3);% (aa1) -- (bbn);
        \draw (aa2) -- (bb3);% (aa2) -- (bbn);
        \draw (aa3) -- (bb2);% (aa3) -- (bbn);
        \node[below=0.25cm of aa2] {\huge{$B$}};

        \node[node_style_a, below right=4cm and 4cm of a3] (aaa1) {$a_1$};
        \node[node_style_a,right=of aaa1] (aaa2) {$a_2$};
        \node[node_style_a,right=of aaa2] (aaa3) {$a_3$};
        \node[node_style_b, above=of aaa1] (bbb1) {$b_1$};
        \draw (aaa2) -- (bbb1) (aaa3) -- (bbb1);
        \node[below=0.25cm of aaa2] {\huge{$C$}};

        \node[node_style_a,right=12cm of a3] (aaaa1) {$h(a_1)$};
        \node[node_style_a,right=of aaaa1] (aaaa2) {$h(a_2)$};
        \node[node_style_a,right=of aaaa2] (aaaa3) {$h(a_3)$};
        \node[node_style_b,above=of aaaa1] (bbbb1) {$h(b_1)$};
        \node[node_style_b,above=of aaaa2] (bbbb2) {$g(b_2)$};
        \node[node_style_b,above=of aaaa3] (bbbb3) {$g(b_3)$};
        \draw (aaaa1) -- (bbbb2) (aaaa1) -- (bbbb3);
        \draw (aaaa2) -- (bbbb1) (aaaa2) -- (bbbb3);
        \draw (aaaa3) -- (bbbb1) (aaaa3) -- (bbbb2);
        \node[below=0.25cm of aaaa2] {\huge{$D$}};

        \draw[->] (a1) -- (aa1) node[midway, above left=0.25cm and 0.25cm] {\huge{$e$}};
        \draw[->] (a2) -- (aa2);
        \draw[->] (a3) -- (aa3);
        \draw[->] (a1) -- (aaa1) node[midway, below left=0.25cm and 0.25cm] {\huge{$f$}};
        \draw[->] (a2) -- (aaa2);
        \draw[->] (a3) -- (aaa3);
        \draw[dotted, ->] (aa1) -- (aaaa1);
        \draw[dotted, ->] (aa2) -- (aaaa2);
        \draw[dotted, ->] (aa3) -- (aaaa3);
        \draw[dotted, ->] (bb2) -- (bbbb2);
        \draw[dotted, ->] (bb3) -- (bbbb3) node[midway, above right=0.25cm and 0.25cm] {\huge{$g$}};
        \draw[dotted, ->] (aaa1) -- (aaaa1);
        \draw[dotted, ->] (aaa2) -- (aaaa2);
        \draw[dotted, ->] (aaa3) -- (aaaa3) node[midway, below right=0.25cm and 0.25cm] {\huge{$h$}};
        \draw[dotted, ->] (bbb1) -- (bbbb1);
    \end{tikzpicture}}
    \caption{The failure of the AP in $\mathcal{PO}_2$.}
    \end{figure}
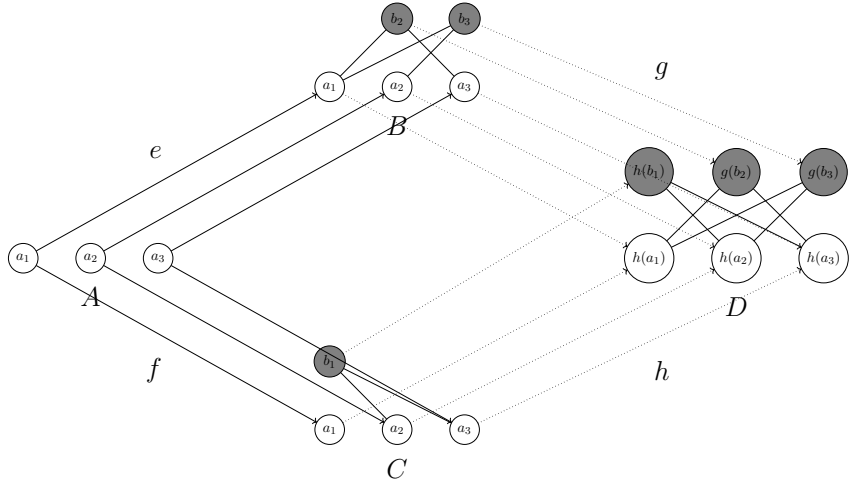

    The partial orders $(A,<)$ and $(C,<)$ are clearly $2$-dimensional, while $(B,<)$ is $n$-dimensional by Hiraguchi's inequality (\ref{eqn:hiraguchi}). Suppose that a partial order $(D,<)$ amalgamates $(B,<)$ and $(C,<)$ over $(A,<)$ via $g$ and $h$, as in Figure \ref{fig:AP}. We claim that the union of the images of $g$ and $h$ forms an embedded copy $(C_{n+1},<)$ in $(D,<)$. 
    
    Since $a_i<b_1$ if and only if $i\neq 1$ in $(C,<)$, $g(a_i)=h(a_i)<h(b_1)$ if and only if $i\neq 1$ in $(D,<)$. Then, since $g(a_1)=h(a_1)\not< h(b_1)$ and $g(a_1)<g(b_j)$ for any $j\neq 1$, it cannot be that $g(b_j)\leq h(b_1)$. Finally, if $h(b_1)<g(b_j)$ for some $j\neq 1$, then since $g(a_j)=h(a_j)<h(b_1)$, we have $g(a_j)<g(b_j)$, contrary to the fact that $a_j\not<b_j$ in $(B,<)$. Thus, the images of $g$ and $h$ form an embedded copy of $(C_{n+1},<)$. But then, $(D,<)$ has dimension $\geq n+1$, showing that the amalgamation property fails.
\end{proof}

We also note that neither $\mathbf{D}_n=(D_n,<)$ nor $(\mathbb{Q}^n,<)$ are ultrahomogeneous.

\begin{proposition}\label{prop:D_n_not_hom}
    For any $n\geq 2$, $\mathbf{D}_n$ and $(\mathbb{Q}^n,<)$ are not ultrahomogeneous.
\end{proposition}

\begin{proof}
    We will describe the proof for $\mathbf{D}_n$, the same argument works for $(\mathbb{Q}^n,<)$. Consider points $\mathbf{a}$, $\mathbf{b}$, and $\mathbf{c}$ in $D_n$ such that 
    \[
        a_1<b_1<c_1 \quad\text{and}\quad c_i<b_i<a_i
    \]
    for $2\leq i\leq n$. Then, $\{\mathbf{a},\mathbf{b},\mathbf{c}\}$ forms an antichain in $\mathbf{D}_n$, and so any permutation of $\{\mathbf{a},\mathbf{b},\mathbf{c}\}$ is an automorphism of $(\{\mathbf{a},\mathbf{b},\mathbf{c}\},<)$. Let $f$ be the permutation which exchanges $\mathbf{a}$ and $\mathbf{b}$, i.e., $f(\mathbf{a})=\mathbf{b}$, $f(\mathbf{b})=\mathbf{a}$, and $f(\mathbf{c})=\mathbf{c}$.
    
    \begin{figure}[ht]
    \scalebox{.67}{
    \begin{tikzpicture}
        \tikzstyle{node_style}=[shape=circle,minimum size=0.01cm,very thin, draw=black, fill=black]
        \tikzstyle{red}=[shape=circle,minimum size=0.01cm,very thin, draw=gray, fill=gray]
        \draw[->] (0,0) -- (0,5);
        \draw[->] (0,0) -- (5,0);
        \draw[->] (8,0) -- (13,0);
        \draw[->] (8,0) -- (8,5);
        \node[red] (x) at (4,3) {};
        \node[left=3 pt of {(4,3)}, outer sep=3pt, fill=white] {$\mathbf{x}$};
        \node[red] (y) at (12,5) {};
        \node[left=3 pt of {(12,5)}, outer sep=3pt, fill=white] {$\mathbf{y}$};
        \node[node_style] (a) at (1,4) {};
        \node[left=3 pt of {(1,4)}, outer sep=3pt,fill=white] {$\mathbf{a}$};
        \node[node_style] (b) at (2,2) {};
        \node[left=3 pt of {(2,2)}, outer sep=3pt,fill=white] {$\mathbf{b}$};
        \node[node_style] (c) at (3,1) {};
        \node[left=3 pt of {(3,1)}, outer sep=3pt,fill=white,] {$\mathbf{c}$};
        \node[node_style] (d) at (9,4) {};
        \node[right=3 pt of {(9,4)}, outer sep=3pt,fill=white] {$f(\mathbf{b})$};
        \node[node_style] (e) at (10,2) {};
        \node[right=3 pt of {(10,2)}, outer sep=3pt,fill=white] {$f(\mathbf{a})$};
        \node[node_style] (f) at (11,1) {};
        \node[right=3 pt of {(11,1)}, outer sep=3pt,fill=white] {$f(\mathbf{c})$};

        \draw[thin,->] (a) -- (e);
        \draw[thin,->] (b) -- (d);
        \draw[thin,->] (c) -- (f);
        \draw[thin,dotted,->] (x) -- (y);
    \end{tikzpicture}}
    \caption{$\mathbf{D}_2$ is not ultrahomogeneous.}
    \end{figure}
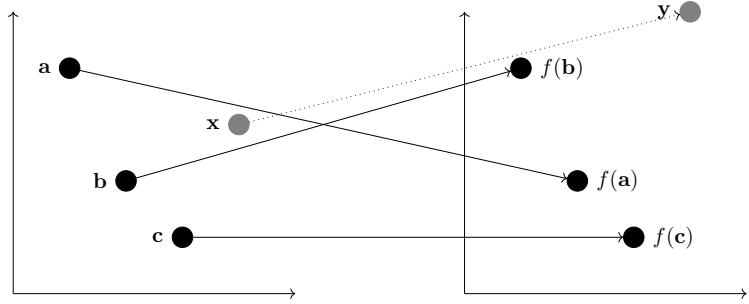

    Consider a new point $\mathbf{x}$ in $D_n$ such that $b_i<x_i$ and $c_i<x_i$ for all $i\leq n$, but $x_2<a_2$. In particular, $\mathbf{b}<\mathbf{x}$ and $\mathbf{c}<\mathbf{x}$, but $\mathbf{a}$ and $\mathbf{x}$ are incomparable. In order to extend $f$ to an isomorphism $g$ on $\{\mathbf{a},\mathbf{b},\mathbf{c},\mathbf{x}\}$, we must assign $g(\mathbf{x})$ to a point $\mathbf{y}$ in $D_n$ where $\mathbf{a}=f(\mathbf{b})<\mathbf{y}$ and $\mathbf{c}=f(\mathbf{c})<\mathbf{y}$. Then, $a_i<y_i$ and $c_i<y_i$ for all $i\leq n$, which implies, by our choice of $\mathbf{a}$, $\mathbf{b}$, and $\mathbf{c}$, that $b_i<y_i$ for all $i\leq n$, i.e., $\mathbf{b}=f(\mathbf{a})<\mathbf{y}$. But $\mathbf{a}$ and $\mathbf{x}$ are incomparable, so we cannot have $g(\mathbf{a})=f(\mathbf{a})<\mathbf{y}=g(\mathbf{x})$. Thus, there is no such $g$.
\end{proof}

A structure is called \emph{$k$-homogeneous} if every isomorphism between substructures of size $\leq\!k$ extends to an automorphism, see \cite{MR1261211}. The preceding proof shows that $\mathbf{D}_n$ is not $3$-homogeneous, but we will see in Corollary \ref{cor:D_n_2-hom} that it is $2$-homogeneous.

\section{\texorpdfstring{$\mathbf{D}_{n,<_1,\ldots,<_n}$}{Dn<1,...,<n} as a Fra\"iss\'e limit}\label{sec:limit}

Throughout this section and the remainder of the article, we fix an $n\geq 2$. We will show that the class $\mathcal{PO}_{n,<_1,\ldots,<_n}$ of all finite $n$-dimensional partial orders with realizers is a Fra\"iss\'e class and that $\mathbf{D}_{n,<_1,\ldots,<_n}=(D_n,<,<_1,\ldots,<_n)$ is its Fra\"iss\'e limit. We start with the following lemma which allows us to carry out back-and-forth arguments involving $\mathbf{D}_{n,<_1,\ldots,<_n}$.

\begin{lemma}\label{lem:forth}
    Let $\mathbf{P}\subseteq\mathbf{Q}$ be structures in $\mathcal{PO}_{n,<_1,\ldots,<_n}$. If $f$ is an embedding of $\mathbf{P}$ into $\mathbf{D}_{n,<_1,\ldots,<_n}$, then there is an embedding $g$ of $\mathbf{Q}$ into $\mathbf{D}_{n,<_1,\ldots,<_n}$ which extends $f$.
    \[
        \begin{tikzcd}
            \mathbf{P} \arrow[r,"f"] \arrow[d,swap,"\subseteq"] & \mathbf{D}_{n,<_1,\ldots,<_n}\\
            \mathbf{Q} \arrow[ur,dotted,swap,"g"] &
        \end{tikzcd}
    \]
\end{lemma}

\begin{proof}
    Suppose $\mathbf{P}=(P,<,<_1,\ldots,<_n)\subseteq\mathbf{Q}=(Q,<,<_1,\ldots,<_n)$ and $f$ is as described. It suffices to consider the case when $|Q|=|P|+1$, as the general situation may be handled by repeated application of this ``one-step'' case. So, we will assume that $P=\{p_1,\ldots,p_m\}$ and $Q=P\cup\{q\}$, and write $\mathbf{a}^j=(a^j_1,\ldots,a^j_n)=f(p_j)\in D_n$ for each $j\leq m$.

    For each $i\leq n$, $q$ lies in a unique interval with respect to the ordering $<_i$ on the elements of $P$: either (i) $q<_i p_j$ for all $j\leq m$, (ii) $p_j<_i q$ for all $j\leq m$, or (iii) there are $j,k\leq m$ such that $p_j<_iq<_ip_k$ with no other elements $<_i$-between $p_j$ and $p_k$. We define a non-empty open interval $I_i$ in $\mathbb{Q}$ according to which of these three cases holds:
    \[
        I_i = \begin{cases} (-\infty,\min\{a^j_i:j\leq m\}) &\text{if (i) holds},\\ (\max\{a^j_i:j\leq m\},\infty) &\text{if (ii) holds},\\ (a^j_i,a^k_i)&\text{if (iii) holds}.\end{cases}
    \]

    Put $U=I_1\times\cdots\times I_n$, a non-empty open hypercube in $\mathbb{Q}^n$. By density, we can choose an element $\mathbf{y}\in U\cap D_n$ and then extend $f$ to $g$ on $Q$ by setting $g(q)=\mathbf{y}$. It remains to verify that $g$ is an embedding. Since $<_1,\ldots,<_n$ realize $<$ on $Q$, it suffices to show that for each $i\leq n$, $g$ preserves $<_i$.

    Fix $i\leq n$ and consider each of the cases (i), (ii), and (iii) above. If (i) holds, $q<_i p_j$ for all $j\leq m$, then 
    \[
        y_i\in I_i=(-\infty,\min\{a^j_i:j\leq m\}),
    \]
    so $y_i<a^j_i$, and thus $\mathbf{y}<_i\mathbf{a}^j=f(p_j)$, for all $j\leq m$. If (ii) holds, then similarly we have that $f(p_j)=\mathbf{a}^j<_i\mathbf{y}$ for all $j\leq m$. If (iii) holds, $p_j<_iq<_i p_k$ for $j,k\leq m$ with no other elements $<_i$-between them, then 
    \[
        y_i\in I_i=(a^j_i,a^j_k),
    \]
    so $a^j_i<y_i<a^k_i$, and thus $f(p_j)=\mathbf{a}^j<_i\mathbf{y}<_i\mathbf{a}^k=f(p_k)$. Hence, $\mathbf{y}$ has exactly the same relationship to $\mathbf{a}^1,\ldots,\mathbf{a}^m$ with respect to $<_i$ as $q$ does to $p_1,\ldots,p_m$, proving that $g$ preserves $<_i$. Thus, $g$ is an embedding.    
\end{proof}

\begin{theorem}\label{thm:D_n_Fraisse}
   $\mathbf{D}_{n,<_1,\ldots,<_n}$ is ultrahomogeneous and its age is $\mathcal{PO}_{n,<_1,\ldots,<_n}$. Consequently, $\mathcal{PO}_{n,<_1,\ldots,<_n}$ is a Fra\"iss\'e class with limit $\mathbf{D}_{n,<_1,\ldots,<_n}$. 
\end{theorem}

\begin{proof}
    To show that an arbitrary structure in $\mathcal{PO}_{n,<_1,\ldots,<_n}$ embeds into $\mathbf{D}_{n,<_1,\ldots,<_n}$, let $\mathbf{Q}$ be such a structure, $\mathbf{P}$ a one-point substructure of $\mathbf{Q}$, and apply Lemma \ref{lem:forth}. Thus, $\mathrm{Age}(\mathbf{D}_{n,<_1,\ldots,<_n})=\mathcal{PO}_{n,<_1,\ldots,<_n}$.

    That $\mathbf{D}_{n,<_1,\ldots,<_n}$ is ultrahomogeneous now follows from Lemma \ref{lem:forth} and Lemma 7.1.4(b) in \cite{MR1221741} via a back-and-forth argument, which we briefly describe here: Suppose that $f_0$ is an isomorphism from finite substructures $\mathbf{P_0}=(P,<,<_1,\ldots,<_n)$ and $\mathbf{Q}_0=(Q,<,<_1,\ldots,<_n)$ of $\mathbf{D}_{n,<_1,\ldots,<_n}$. Enumerate $D_n\setminus P$ as $\{\mathbf{a}^i:i\in\mathbb{N}\}$ and $D_n\setminus Q$ as $\{\mathbf{b}^i:i\in\mathbb{N}\}$. We recursively define an increasing chain of functions extending $f_0$. By Lemma \ref{lem:forth}, going ``forth'', there is an extension $f_1$ of $f_0$ to $\mathbf{P}_1=(P\cup\{\mathbf{a}^0\},<,<_1,\ldots,<_n)$ whose range $\mathbf{Q}_1$ is a substructure of $\mathbf{D}_{n,<_1,\ldots,<_n}$ which contains $\mathbf{Q}$. Next, going ``back'', there is an extension $f_2^{-1}$ of $f_1^{-1}$ to $\mathbf{Q}_2=(Q\cup\{f_1(\mathbf{a}^0),\mathbf{b}^0\},<,<_1,\ldots,<_n)$ whose range $\mathbf{P}_2$ contains $\mathbf{P}_1$. We continue in this fashion. The resulting $f=\bigcup_{m\in\mathbb{N}}f_m$ is then an automorphism of $\mathbf{D}_{n,<_1,\ldots,<_n}$ which extends $f_0$. It now follows from Theorem \ref{thm:Fraisse2} that $\mathcal{PO}_{n,<_1,\ldots,<_n}$ is a Fra\"iss\'e class and $\mathbf{D}_{n,<_1,\ldots,<_n}$ its limit.
\end{proof}

Lemma \ref{lem:forth} also implies the following (see Lemma 7.1.3 in \cite{MR1221741}):

\begin{corollary}
    Every countable $n$-dimensional partial order with realizers embeds into $\mathbf{D}_{n,<_1,\ldots,<_n}$.\qed
\end{corollary}

We note here that $(\mathbb{Q}^n,<,<_{\mathrm{lex}_1},\ldots,<_{\mathrm{lex}_n})$ is not ultrahomogeneous:

\begin{proposition}\label{prop:Qn_not_hom}
    $(\mathbb{Q}^n,<,<_{\mathrm{lex}_1},\ldots,<_{\mathrm{lex}_n})$ is not ultrahomogeneous.
\end{proposition}

\begin{proof}
    Consider points $\mathbf{a},\mathbf{b},\mathbf{c},\mathbf{x}\in\mathbb{Q}^n$ such that 
    \[
        a_1<x_1<b_1<c_1 \quad \text{and}\quad a_i<x_i<c_i<b_i
    \]
    for $2\leq i\leq n$, so that 
    \[
        \mathbf{a}<_{\mathrm{lex}_1}\mathbf{x}<_{\mathrm{lex}_1}\mathbf{b}<_{\mathrm{lex}_1}\mathbf{c}\quad\text{and}\quad\mathbf{a}<_{\mathrm{lex}_i}\mathbf{x}<_{\mathrm{lex}_i}\mathbf{c}<_{\mathrm{lex}_i}\mathbf{b}
    \]
    for $2\leq i\leq n$. 

    \begin{figure}[ht]
    \scalebox{.67}{
    \begin{tikzpicture}
        \tikzstyle{node_style}=[shape=circle,minimum size=0.01cm,very thin, draw=black, fill=black]
        \tikzstyle{red}=[shape=circle,minimum size=0.01cm,very thin, draw=gray, fill=gray]
        \draw[->] (0,0) -- (0,5);
        \draw[->] (0,0) -- (5,0);
        \draw[->] (8,0) -- (13,0);
        \draw[->] (8,0) -- (8,5);
        \node[red] (x) at (1.5,1.5) {};
        \node[left=3 pt of {(1.5,1.5)}, outer sep=3pt, fill=white] {$\mathbf{x}$};
        \node[node_style] (a) at (1,1) {};
        \node[left=3 pt of {(1,1)}, outer sep=3pt,fill=white] {$\mathbf{a}$};        
        \node[node_style] (b) at (2,4) {};
        \node[left=3 pt of {(2,4)}, outer sep=3pt,fill=white] {$\mathbf{b}$};
        \node[node_style] (c) at (4,2) {};
        \node[left=3 pt of {(4,2)}, outer sep=3pt,fill=white,] {$\mathbf{c}$};
        \node[node_style] (d) at (9,1) {};
        \node[right=3 pt of {(9,1)}, outer sep=3pt,fill=white] {$\mathbf{a}'$};
        \node[node_style] (e) at (9,4) {};
        \node[right=3 pt of {(9,4)}, outer sep=3pt,fill=white] {$\mathbf{b}'$};
        \node[node_style] (f) at (12,1) {};
        \node[right=3 pt of {(12,1)}, outer sep=3pt,fill=white] {$\mathbf{c}'$};

        \draw[thin,->] (a) -- (d);
        \draw[thin,->] (b) -- (e);
        \draw[thin,->] (c) -- (f);
        \draw[thin,dotted,->] (x) -- (d);
    \end{tikzpicture}}
    \caption{$(\mathbb{Q}^2,<,<_{\mathrm{lex}_1},<_{\mathrm{lex}_2})$ is not ultrahomogeneous.}
    \end{figure}
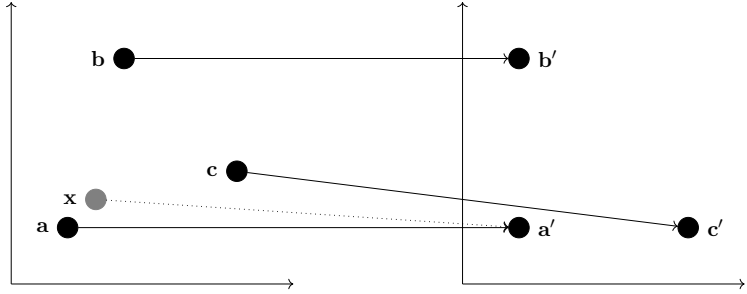
    
    Next, take points $\mathbf{a}',\mathbf{b}',\mathbf{c}'\in\mathbb{Q}^n$ such that
    \[
        a_1'=b_1'<c_1\quad\text{and}\quad a_i'=c_i'<b_i'
    \]
    for $2\leq i\leq n$, so that
    \[ \mathbf{a}'<_{\mathrm{lex}_1}\mathbf{b}'<_{\mathrm{lex}_1}\mathbf{c}'\quad\text{and}\quad\mathbf{a}'<_{\mathrm{lex}_i}\mathbf{c}'<_{\mathrm{lex}_i}\mathbf{b}'
    \]
    for $2\leq i\leq n$. The mapping $f$ given by $f(\mathbf{a})=\mathbf{a}'$, $f(\mathbf{b})=\mathbf{b}'$ and $f(\mathbf{c})=\mathbf{c}'$ is then an isomorphism from $(\{\mathbf{a},\mathbf{b},\mathbf{c}\},<,<_{\mathrm{lex}_1},\ldots,<_{\mathrm{lex}_n})$ to $(\{\mathbf{a}',\mathbf{b}',\mathbf{c}'\},<,<_{\mathrm{lex}_1},\ldots,<_{\mathrm{lex}_n})$. But, we claim there is no way to extend $f$ to an automorphism $g$ on all of $\mathbb{Q}^n$: If $g$ was such an automorphism, then $g(\mathbf{x})=\mathbf{x}'$ must satisfy:
    \[  
\mathbf{a}'<_{\mathrm{lex}_1}\mathbf{x}'<_{\mathrm{lex}_1}\mathbf{b}'\quad\text{and}\quad\mathbf{a}'<_{\mathrm{lex}_i}\mathbf{x}'<_{\mathrm{lex}_i}\mathbf{c}'
    \] 
    for $2\leq i\leq n$. The first of these implies $a_1'=x_1'=b_1'$, while the second implies $a_i'=x_i'=c_i'$ for $2\leq i\leq n$, and so $\mathbf{a}'=\mathbf{x}'$, a contradiction.
\end{proof}

\section{A finite axiomatization for \texorpdfstring{$\mathbf{D}_{n,<_1,\ldots,<_n}$}{Dn<1,...,<n)}}\label{sec:axioms}

Theorems \ref{thm:Fraisse1} and \ref{thm:D_n_Fraisse} tell us that $\mathbf{D}_{n,<_1,\ldots,<_n}$ is a canonical object; it is the unique countable ultrahomogeneous $n$-dimensional partial order with realizers, up to isomorphism, in which all finite $n$-dimensional partial orders with realizers embed. But there is another sense in which it is canonical, namely that any countable structure, in the same signature, that satisfies the same first-order sentences must be isomorphic to it. In other words, its \emph{theory} is \emph{$\aleph_0$-categorical}. This is because, for any $m\in\mathbb{N}$, whether one $m$-tuple from $D_n$ can be mapped to another by an automorphism of $\mathbf{D}_{n,<_1,\ldots,<_n}$ is entirely determined by the finitely many order relations among the coordinates of the $m$-tuples (this can be seen in the proof of Lemma \ref{lem:forth}). In other words, $\mathrm{Aut}(\mathbf{D}_{n,<_1,\ldots,<_n})$ is \emph{oligomorphic}; it has finitely many orbits on $D_n^m$ for any $m\in\mathbb{N}$. It then follows by the theorem of Engeler, Ryll-Nardzewski, and Svenonius (Theorem 7.3.1 in \cite{MR1221741}) that the theory of $\mathbf{D}_{n,<_1,\ldots,<_n}$ is $\aleph_0$-categorical.

In what follows, we will show a stronger result. We will describe a finite list of axioms, which are clearly satisfied by $\mathbf{D}_{n,<_1,\ldots,<_n}$, and then prove that they are $\aleph_0$-categorical. It was shown in \cite{MR0641492} that the reduct $\mathbf{D}_n=(D_n,<)$ has a finitely axiomatizable $\aleph_0$-categorical theory, but the added expressive power of the realizers affords us a conceptually simpler axiomatization for the expansion $\mathbf{D}_{n,<_1,\ldots,<_n}$. %This axiomatization is based on the axiomatization of $\mathbf{D}_2=(\mathbb{D}_2,<)$ given in \cite{MR0641492}, though is in some ways simplified, due to the additional expressive power afforded to us by the conjugate order.

First, some terminology: If $<_1,\ldots,<_n$ are linear orders on a set $P$, a point $p\in P$ \emph{determines} $2n$ \emph{regions}, via the formulas (with free variable $u$):
\[
    u<_i p \quad\text{and}\quad  p<_iu
\]
for $i\leq n$. More generally, finitely many points in $P$ \emph{determine} finitely many \emph{regions} as intersections (i.e., conjunctions of the formulas above) of the regions determined by each point. We call such regions \emph{non-trivial} if the corresponding region in $D_n$, defined by the same formula, is non-empty.

\begin{example}\label{ex:6_points}
    Consider the six points lying in $D_2$ pictured below. The vertical and horizontal dotted lines through each point divide $D_2$ into the four regions determined by that point. Collectively, these six points determine non-trivial 49 regions. As each such region is an open rectangle (possibly unbounded on one or more sides), it is uniquely determined by at most four points.
    
\begin{figure}[ht]
\scalebox{0.5}{
\begin{tikzpicture}
    \draw[very thick][<->] (0,8) -- (0,-4);
    \draw[very thick][<->] (8,0) -- (-2,0);
    
    \tikzstyle{node_style}=[shape=circle,minimum size=0.01cm,very thin, draw=black, fill=black]
    \tikzstyle{node_outline}=[shape=circle,minimum size=0.5cm,very thin, draw=black]
    \tikzstyle{nodegray}=[shape=circle,minimum size=0.01cm,very thin, draw=gray, fill=gray]

    \newcounter{m}
    \newcounter{n}
    \foreach \m in {3,1,4,-.5,1.5,7}{
    \draw[thick, dotted, gray] (\m,8) -- (\m,-4);
     }
     \foreach \n in {2,3.5,6,-1,-3,-1.5}{
    \draw[thick, dotted, gray] (8,\n) -- (-2,\n);
     }
    
    \node[node_style] (i1) at (3,2) {};
    \node[node_style] (j1) at (3,2) {};
    \node[node_style] (i2) at (1,3.5) {};
    \node[node_style] (j2) at (1,3.5) {};
    \node[node_style] (i3) at (4,6) {};
    \node[node_style] (j3) at (4,6) {};
    \node[node_style] (i4) at (-.5,-1) {};
    \node[node_style] (i5) at (1.5,-3) {};
    \node[node_style] (i6) at (7,-1.5) {};
    \node[node_style] (j6) at (7,-1.5) {};
\end{tikzpicture}}
\caption{Six points and the regions they determine in $\mathbf{D}_{2,<_1,<_2}$.}
\end{figure}
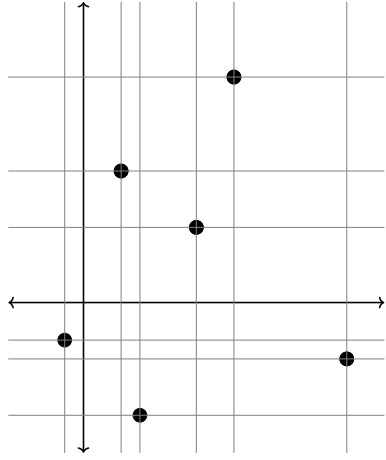   
\end{example}

The following axioms for a structure of the form $(P,<,<_1,\ldots,<_n)$ say that $<$ is a partial order, $<_1,\ldots,<_n$ are linear orders which realize $<$, and every non-trivial region determined by finitely many points in $P$ contains a point. Since such regions in $D_n$ are (possibly unbounded) $n$-dimensional open hypercubes, they are uniquely determined by at most $2^n$ points.

\begin{definition}
    The theory $\mathsf{DPO}_{n,<_1,\ldots,<_n}$ of \emph{dense $n$-dimensional partial orders with realizers} consists of the following sentences:
    \begin{itemize}
        \item $<$ is a partial ordering:
        \begin{align*}
            &\forall x(x\not< x)\\
            %&\forall x\forall y(x<y\rightarrow y\not< x)\\
            &\forall x\forall y\forall z(x<y\wedge y<z \rightarrow x<z).
        \end{align*}
        \item Each of $<_1,\ldots,<_n$ is a linear ordering:
        \begin{align*}
            &\forall x(x\not<_ix)\\
            %&\forall x\forall y(x<_iy\rightarrow y\not<_i x)\\
            &\forall x\forall y\forall z(x<_iy\wedge y<_i z \rightarrow x<_i z)\\
            &\forall x\forall y(x<_i y \vee x=y \vee y<_i x)
        \end{align*}
        for $i\leq n$.
        \item $<_1,\ldots,<_n$ realize $<$:
        \begin{align*}
            \forall x\forall y(x<y \leftrightarrow (x<_1y\wedge\cdots\wedge x<_n y)).
        \end{align*}
        \item $<_1,\ldots,<_n$ are dense and unbounded:
        \begin{align*}
            \forall x_1\cdots\forall x_{2^n} 
            \left[\bigwedge_{i=1}^{\#(n)}\left(\phi_i(x_1,\ldots,x_{2^n})\rightarrow
            \bigwedge_{j=1}^{r(i)}\psi^i_j(x_1,\ldots,x_{2^n})\right)\right],
        \end{align*}
        where:
        \begin{itemize}
            \item  $\#(n)$ is the number of possible configurations of $\leq 2^n$ many points with respect to the linear orders $<_1,\ldots,<_n$, 
            \item for each $i\leq\#(n)$, $\phi_i(x_1,\ldots,x_{2^n})$ asserts that $x_1,\ldots,x_{2^n}$ are in the $i$th such configuration (for configurations of fewer than $2^n$ points, we may assert that some of the $x_k$ are equal),
            \item $r(i)$ is the number of non-trivial regions determined by $x_1,\ldots,x_{2^n}$ in the $i$th configuration, and
            \item for each $j\leq r(i)$, $\psi^i_j(x_1,\ldots,x_{2^n})$ asserts the existence of a point the $j$th non-trivial region determined by $x_1,\ldots,x_{2^n}$. 
        \end{itemize}
    \end{itemize}
\end{definition}

It is clear that $\mathbf{D}_{n,<_1,\ldots,<_n}$ satisfies these axioms. Meanwhile, that $(\mathbb{Q}^n,<,<_{\mathrm{lex}_1},\ldots,<_{\mathrm{lex}_n})$ does not can be seen from the proof of Proposition \ref{prop:Qn_not_hom}; the region determined by the points $\mathbf{a}'$, $\mathbf{b}'$, and $\mathbf{c}'$ therein is empty.

\begin{example}
    Consider the same six points lying in $D_2$ described in Example \ref{ex:6_points} above, four of which we label as $\mathbf{x}$, $\mathbf{y}$, $\mathbf{z}$, and $\mathbf{w}$ as below:
    
\begin{figure}[ht]
\scalebox{0.5}{
\begin{tikzpicture}
    \draw[very thick][<->] (0,8) -- (0,-4);
    \draw[very thick][<->] (8,0) -- (-2,0);
    
    \tikzstyle{node_style}=[shape=circle,minimum size=0.01cm,very thin, draw=black, fill=black]
    \tikzstyle{node_outline}=[shape=circle,minimum size=0.5cm,very thin, draw=black]
    \tikzstyle{nodegray}=[shape=circle,minimum size=0.01cm,very thin, draw=gray, fill=gray]
    
    \newcounter{i}
    \newcounter{j}
    \foreach \i in {3,1,4,-.5,1.5,7}{
    \draw[thick, dotted, gray] (\i,8) -- (\i,-4);
     }
     \foreach \j in {2,3.5,6,-1,-3,-1.5}{
    \draw[thick, dotted, gray] (8,\j) -- (-2,\j);
     }
     %\pgfsetfillopacity{0.2}
     \draw [very thick, dotted, fill=lightgray] (4,2) rectangle (7,3.5);

    \node[node_style] (i1) at (3,2) {};
    %\node[node_outline] (j1) at (3,2) {};
    \node[below left =5 pt of {(3,2)}, outer sep=3pt] {\LARGE{$\mathbf{y}$}};
    \node[node_style] (i2) at (1,3.5) {};
    %\node[node_outline] (j2) at (1,3.5) {};
    \node[below left =5 pt of {(1,3.5)}, outer sep=3pt] {\LARGE{$\mathbf{x}$}};
    \node[node_style] (i3) at (4,6) {};
    %\node[node_outline] (j3) at (4,6) {};
    \node[below left =5 pt of {(4,6)}, outer sep=3pt] {\LARGE{$\mathbf{z}$}};
    \node[node_style] (i4) at (-.5,-1) {};
    \node[node_style] (i5) at (1.5,-3) {};
    \node[node_style] (i6) at (7,-1.5) {};
    %\node[node_outline] (j6) at (7,-1.5) {};
    \node[below left =5 pt of {(7,-1.5)}, outer sep=3pt] {\LARGE{$\mathbf{w}$}};
    \node[nodegray] (r) at (6,3) {};
    \node[below left =5 pt of {(6,3)}, outer sep=3pt] {\LARGE{$u$}};
\end{tikzpicture}}
\end{figure}

    The points $\mathbf{x}$, $\mathbf{y}$, $\mathbf{z}$, and $\mathbf{x}$ lie in the $i$th (say) configuration described by the formula $\phi_i(\mathbf{x},\mathbf{y},\mathbf{z},\mathbf{w})$:
    \[
        \mathbf{x}<_1\mathbf{y}\wedge\mathbf{y}<_1\mathbf{z} \wedge \mathbf{y}<_1\mathbf{z} \wedge \mathbf{z}<_2 \mathbf{y} \wedge \mathbf{y}<_2\mathbf{x} \wedge \mathbf{x}<_2\mathbf{z}.
    \]
    These points determine the $j$th (say) shaded region (in free variable $u$):
    \[
        \mathbf{z}<_1 u \wedge u<_1\mathbf{w} \wedge \mathbf{y}<_2 u \wedge u <_2\mathbf{x}.
    \]
    Then, $\psi^i_j(\mathbf{x},\mathbf{y},\mathbf{z},\mathbf{w})$ asserts the existence of a point $u$ in this region:
    \[
        \exists u(\mathbf{z}<_1 u \wedge u<_1\mathbf{w} \wedge \mathbf{y}<_2 u \wedge u <_2\mathbf{x}).
    \]
\end{example}

The next lemma is an analogue of Lemma \ref{lem:forth} for arbitrary dense $n$-dimensional partial orders with realizers.

\begin{lemma}\label{lem:back}
    Suppose that $\mathbf{D}$ satisfies $\mathsf{DPO}_{n,<_1,\ldots,<_n}$. If $\mathbf{P}\subseteq\mathbf{Q}$ are structures in $\mathcal{PO}_{n,<_1,\ldots,<_n}$ and $f$ is an embedding of $\mathbf{P}$ into $\mathbf{D}$, then there is an embedding $g$ of $\mathbf{Q}$ into $\mathbf{D}$ which extends $f$.
\end{lemma}

\begin{proof}
    Suppose that $\mathbf{D}=(D,<,<_1,\ldots,<_n)$. By Lemma \ref{lem:forth}, we may assume that $\mathbf{P}=(P,<,<_1,\ldots,<_n)$ and $\mathbf{Q}=(Q,<,<_1,\ldots,<_n)$ are substructures of $\mathbf{D}_{n,<_1,\ldots,<_n}$. As in the proof of Lemma \ref{lem:forth}, we may further assume that $P=\{\mathbf{p}_1,\ldots,\mathbf{p}_m\}$, $Q=P\cup\{\mathbf{q}\}$, and write $d_i=f(\mathbf{p}_i)$ for $i\leq m$.

    The points $\mathbf{p}_1,\ldots,\mathbf{p}_m$ define finitely many regions in $\mathbf{D}_{n,<_1,\ldots,<_n}$ as above, and $\mathbf{q}$ lies in a unique such region determined by (at most) $2^n$ of the $\mathbf{p}_i$'s, say $\mathbf{p}_{i_1},\ldots,\mathbf{p}_{i_{2^n}}$ (if fewer than $2^n$-many points are needed, simply repeat points). These points have a configuration with respect to the $<_1,\ldots,<_n$ described by some formula $\phi_i(x_1,\ldots,x_{2^n})$, that is, $\phi_i(\mathbf{p}_{i_1},\ldots,\mathbf{p}_{i_{2^n}})$ holds in $\mathbf{D}_{n,<_1,\ldots,<_n}$. Likewise, $\mathbf{q}$ is a witness to some existential assertion $\psi^i_j(\mathbf{p}_{i_1},\ldots,\mathbf{p}_{i_{2^n}})$. Since $f$ is an embedding, $\phi_i(d_{i_1},\ldots,d_{i_{2^n}})$ holds in $\mathbf{D}$. Then, by the fact that $\mathbf{D}$ satisfies $\mathsf{DPO}_{n,<_1,\ldots,<_n}$, there must be a witness for $\psi^i_j(d_{i_1},\ldots,d_{i_{2^n}})$ in $D$, call it $e$. Extending $f$ to $g$ by declaring $g(\mathbf{q})=e$ will give an embedding of $\mathbf{Q}$ into $\mathbf{D}$, as desired.
\end{proof}

\begin{theorem}\label{thm:axiomatization}
    $\mathbf{D}_{n,<_1,\ldots,<_n}$ is the unique countable dense $n$-dimensional partial order with realizers, up to isomorphism.
\end{theorem}

\begin{proof}
   Given a countable dense $n$-dimensional partial order with realizers $\mathbf{D}$, an isomorphism from $\mathbf{D}$ to $\mathbf{D}_{n,<_1,\ldots,<_n}$ can now be constructed by a back-and-forth argument, using Lemma \ref{lem:forth} to go ``forth'' and Lemma \ref{lem:back} to go ``back'', as in the proof of Theorem \ref{thm:D_n_Fraisse}.
\end{proof}

We end this section with a consequence regarding the homogeneity of the reduct $\mathbf{D}_n=(D_n,<)$:

\begin{corollary}\label{cor:D_n_2-hom}
   $\mathbf{D}_n$ is $2$-homogeneous.
\end{corollary}

\begin{proof}
    Clearly, by Theorem \ref{thm:D_n_Fraisse}, any assignment of one point to another can be extended to an automorphism of $\mathbf{D}_{n,<_1,\ldots,<_n}$, so $\mathbf{D}_n$ is $1$-homogeneous.

    Suppose that $\mathbf{a}$ and $\mathbf{b}$ are distinct points in $D_n$, and $f$ is an isomorphism of $(\{(\mathbf{a},\mathbf{b}\},<)$ sending $\mathbf{a}$ to $\mathbf{a}'$ and $\mathbf{b}$ to $\mathbf{b}'$, also in $D_n$. If $\mathbf{a}'$ and $\mathbf{b}'$ have the same relationships with respect to $<_1,\ldots,<_n$ as $\mathbf{a}$ and $\mathbf{b}$, then $f$ preserves $<_1,\ldots,<_n$, and so extends to an automorphism of $\mathbf{D}_{n,<_1,\ldots,<_n}$, and thus of $\mathbf{D}_n$, by Theorem \ref{thm:D_n_Fraisse}.
    
    Suppose that, for some $i_1,\ldots,i_k\leq n$, $\mathbf{a}'$ and $\mathbf{b}'$ have the opposite relationship with respect to $<_{i_1},\ldots,<_{i_k}$ as $\mathbf{a}$ and $\mathbf{b}$. Observe that $(D_n,<,<_1^{\iota(1)},\ldots,<_n^{\iota(n)})$, where $<_i^{\iota(i)}\;=\;<_i$ if $i\notin\{i_1,\ldots,i_k\}$, and $<_i^{\iota(i)}\;=\;>_i$ if $i\in\{i_1,\ldots,i_k\}$, also satisfies $\mathsf{DPO}_{n,<_1,\ldots,<_n}$, so by Theorem \ref{thm:axiomatization}, is isomorphic to $\mathbf{D}_{n,<_1,\ldots,<_n}$ via some $g$. Hence, $g\circ f$ is an isomorphsim of $(\{\mathbf{a},\mathbf{b}\},<,<_1,\ldots,<_n)$ to some other substructure of $\mathbf{D}_{n,<_1,\ldots,<_n}$, and thus, by Theorem \ref{thm:D_n_Fraisse}, extends to an automorphism $h$ of $\mathbf{D}_{n,<_1,\ldots,<_n}$. But then, $g^{-1}\circ h$ is an automorphism of $\mathbf{D}_n$ which extends $f$.
\end{proof}

\section{The Ramsey property and extreme amenability of \texorpdfstring{$\mathrm{Aut}(\mathbf{D}_{n,<_1,\ldots,<_n})$}{Aut(Dn<1,...,<n)}}\label{sec:Ramsey}

In this section, we will prove a Ramsey-theoretic property of the class $\mathcal{PO}_{n,<_1,\ldots,<_n}$, and then use it to derive a strong fixed-point property for the group $\mathrm{Aut}(\mathbf{D}_{n,<_1,\ldots,<_n})$ of all automorphisms of $\mathbf{D}_{n,<_1,\ldots,<_n}$.

First, we recall a definition from topological dynamics: A topological group $G$ is \emph{extremely amenable} if every continuous action of $G$ on a compact space $X$ (i.e.,~the induced map $G\times X\to X:(g,x)\mapsto g\cdot x$ is continuous) has a fixed point.\footnote{This definition strengthens one of the many characterizations for a locally compact group $G$ to be \emph{amenable}, namely that  every affine action of $G$ on a compact convex set has a fixed point. However, \emph{no} locally compact group is extremely amenable, a consequence of a result of Veech \cite{MR0467705}, see also Appendix 2 of \cite{MR2140630}.} Note that all topological spaces and groups considered here are assumed to be Hausdorff.

The group $\mathrm{Aut}(\mathbf{M})$ of all automorphisms of a countably-infinite structure $\mathbf{M}$ becomes a Polish (i.e.,~separable and complete metrizable) topological group when endowed with the topology of pointwise convergence. In the case that $\mathbf{M}$ is the Fra\"iss\'e limit of a Fra\"iss\'e \emph{order} class $\mathcal{K}$, meaning one in which all of its structures are endowed with a linear order, the work of Kechris, Pestov, and Todor\v{c}evi\'{c} \cite{MR2140630} characterizes extreme amenability of $\mathrm{Aut}(\mathbf{M})$ in terms of a Ramsey-theoretic property of $\mathcal{K}$.

Suppose that $\mathcal{K}$ is a hereditary class of finite structures. Given $\mathbf{A},\mathbf{B}\in\mathcal{K}$ with $\mathbf{A}\preceq \mathbf{B}$, we write
\[
    \binom{\mathbf{B}}{\mathbf{A}}=\{\mathbf{A}_0:\mathbf{A}_0 \subseteq \mathbf{B} \text{ and } \mathbf{A}_0\cong \mathbf{A}\}.
\]
We call elements of $\binom{\mathbf{B}}{\mathbf{A}}$ ``copies of $\mathbf{A}$ in $\mathbf{B}$''. We say that $\mathcal{K}$ has the \emph{Ramsey property} if for any $\mathbf{A}\preceq \mathbf{B}$ in $\mathcal{K}$ and $k\in\mathbb{N}^+$, there is a $\mathbf{C}\in\mathcal{K}$ with $\mathbf{B}\preceq \mathbf{C}$ and such that for any \emph{$k$-colouring} $c:\binom{\mathbf{C}}{\mathbf{A}}\to\{1,\ldots,k\}$, there is a $\mathbf{B}_0\in\binom{\mathbf{C}}{\mathbf{B}}$ which is \emph{monochromatic}, i.e.,~$|c[\binom{\mathbf{B}_0}{\mathbf{A}}]|=1$. 

\begin{theorem}[Kechris--Pestov--Todor\v{c}evi\'{c} \cite{MR2140630}]\label{thm:KPT}
    Let $\mathcal{K}$ be a Fra\"iss\'e order class and $\mathbf{M}$ its Fra\"iss\'e limit. Then, $\mathrm{Aut}(\mathbf{M})$ is extremely amenable if and only if $\mathcal{K}$ has the Ramsey property.
\end{theorem}

As mentioned in the introduction, the classical form of Ramsey's Theorem can be viewed as saying that the class $\mathcal{LO}$ of all finite linear orders has the Ramsey property and thus, by Theorem \ref{thm:KPT}, $\mathrm{Aut}(\mathbb{Q},<)$ is extremely amenable. Likewise, the class $\mathcal{PO}_{\prec}$ of all finite partial orders with a linear extension (i.e., structures $(P,<,\prec)$ where $<$ is a partial order on $P$ and $\prec$ a linear order extending $<$) has the Ramsey property (Theorem 2.4 \cite{MR2128088}, see also \cite{MR0437351}, \cite{MR0818500}, and \cite{MR2948746}), and so the automorphism group of its Fra\"iss\'e limit is also extremely amenable.

However, the class $\mathcal{PO}$ of all finite partial orders does not have the Ramsey property, and neither does the class $\mathcal{PO}_n$ of all finite $n$-dimensional partial orders, as the following example shows:

\begin{example}
    Let $(A,<)$ be the partial order given by $A=\{a,b,c\}$, where $a<b$ and $c$ is incomparable with both $a$ and $b$. Let $(B,<)$ be the product of $(\{1,2,3\},<)$ with itself (i.e., a $\mathbf{3}^2$-grid in the terminology below), so $(A,<)\preceq(B,<)$. Let $(C,<)$ be any finite partial order such that $(B,<)\preceq(C,<)$ and $\prec$ a linear order on $(C,<)$ which extends $<$ and does not satisfy $a\prec c\prec b$. Define a $2$-colouring on the copies of $(A,<)$ in $(C,<)$ based on how the images of $a$ and $c$ are ordered with respect to $\prec$. Then, on any copy of $(B,<)$ in $(C,<)$, both colours must be realized.
    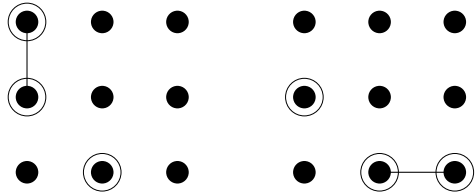
\begin{figure}[ht]
    \centering
    \begin{tikzpicture}
        \draw[white] (0,0) grid (2,2);
        \newcounter{r}
        \newcounter{s}
        \forloop{r}{0}{\value{r}<3}{
        \forloop{s}{0}{\value{s}<3}{
        \fill[black] (\value{r},\value{s}) circle(.15);
        }
        }
        \draw (0,1) -- (0,2);
        \foreach \p in {(0,1),(0,2),(1,0)}
        \node[shape=circle , minimum size=.5cm, very thin, draw=black] () at \p {} ;
    \end{tikzpicture}
    \quad\quad\quad
    \begin{tikzpicture}
        \draw[white] (0,0) grid (2,2);
        \newcounter{u}
        \newcounter{v}
        \forloop{u}{0}{\value{u}<3}{
        \forloop{v}{0}{\value{v}<3}{
        \fill[black] (\value{u},\value{v}) circle(.15);
        }
        }
        \draw (1,0) -- (2,0);
        \foreach \p in {(1,0),(2,0),(0,1)}
        \node[shape=circle , minimum size=.5cm, very thin, draw=black] () at \p {} ;
    \end{tikzpicture}
    \caption{Two different ``colours'' of $(A,<)$ in $(B,<)$.}
    \end{figure}
\end{example}

The main result of this section is the following:

\begin{theorem}\label{thm:PO_n_RP}
   $\mathcal{PO}_{n,<_1,\ldots,<_n}$ has the Ramsey property.
\end{theorem}

In any $n$-dimensional partial order with realizers $(P,<,<_1,\ldots,<_n)$, the order $<$ is definable in terms of the realizers as their intersection, and since an intersection of linear orders is always a partial order, Theorem \ref{thm:PO_n_RP} is equivalent to saying that the class of finite sets endowed with $n$ linear orders has the Ramsey property. It is in this form that Theorem \ref{thm:PO_n_RP} was previously proved by Soki\'c \cite{MR3047085} and, independently, Bodirsky \cite{MR3210656} (see also \cite{MR3743099}, \cite{MR2948746}, \cite{MR2948747}, and \cite{MR3567532} for related results). Both Soki\'c and Bodirsky derive this result as a special case of general theorems concerning ways of combining Ramsey classes together, and ultimately rely on forms of the Product Ramsey Theorem (Theorem 5 in \S5.1 of \cite{MR3288500}).

Our proof of Theorem \ref{thm:PO_n_RP}, instead, applies the Product Ramsey Theorem directly using a straightforward induced colouring argument. It is based on an argument in a recent paper of Bir\'o and Wan \cite{MR4468883}, where they prove a different Ramsey-theoretic property\footnote{The version of the ``Ramsey property'' stated in \cite{MR4468883} is for colourings of the comparability graph of a partial order, rather than substructures.} for $n$-dimensional partial orders, and from which we borrow the following terminology: For $m,n\in\mathbb{N}^+$, we call the structure $\mathbf{m}^n=(\{1,\ldots,m\}^n,<)$, with the product order, an \emph{$\mathbf{m}^n$-grid}, and for any $a_1,\ldots,a_n\subseteq \{1,\ldots,m\}$, the substructure $(a_1\times\cdots\times a_n,<)$ is a \emph{subgrid} of $\mathbf{m}^n$. If, moreover, $|a_1|=\cdots=|a_n|=l$, we call $(a_1\times\cdots\times a_n,<)$ an \emph{$\mathbf{l}^n$-subgrid} of $\mathbf{m}^n$. We will write $<_1,\ldots,<_n$ for the $n$ lexicographic orders on $\mathbf{m}^n$ which realize the product order. As in \cite{MR4468883}, we can now state the Product Ramsey Theorem in the following convenient way:

\begin{theorem}[Product Ramsey Theorem \cite{MR3288500}]\label{thm:prod_Ramsey}
    For all $k,l,m,n\in\mathbb{N}^+$, there is an $r\in\mathbb{N}^+$ such that for any $k$-colouring of the $\mathbf{l}^n$-subgrids of $\mathbf{r}^n$, there is an $\mathbf{m}^n$-subgrid of $\mathbf{r}^n$ all of whose $\mathbf{l}^n$-subgrids get the same colour.
\end{theorem}

In other words, the class of all $n$-dimensional grids has the Ramsey property, with respect to embeddings as subgrids.

Given a finite $n$-dimensional partial order with realizers $\mathbf{P}$ and $m=|P|$, a \emph{rigid embedding} of $\mathbf{P}$ into $\mathbf{m}^n$ is an embedding $f$ of $\mathbf{P}$ into the structure $\mathbf{m}^n_{<_1,\ldots,<_n}=(\{1,\ldots,m\}^n,<,<_1,\ldots,<_n)$ such that no two distinct points in the image of $f$ share a common coordinate, i.e.,~the image of $f$ contains no colinear points. These were called ``casual embeddings'' in \cite{MR4468883}, with the orders $<_1,\ldots,<_n$ omitted.

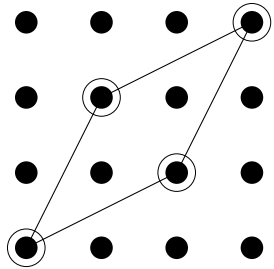
\begin{figure}[ht]
    \centering
    \begin{tikzpicture}
        \draw[white] (0,0) grid (3,3);
        \newcounter{x}
        \newcounter{y}
        \forloop{x}{0}{\value{x}<4}{
        \forloop{y}{0}{\value{y}<4}{
        \fill[black] (\value{x},\value{y}) circle(.15);
        }
        }
        \draw (0,0) -- (1,2) -- (3,3) -- (2,1) -- (0,0);
        \foreach \p in {(0,0),(1,2),(2,1),(3,3)}
        \node[shape=circle , minimum size=.5cm, very thin, draw=black] () at \p {} ;
    \end{tikzpicture}
    \caption{A rigid embedding into a $\mathbf{4}^2$-grid}
\end{figure}

Our terminology reflects the following fact (cf.~Lemma 6.4 of \cite{MR4468883}):

\begin{lemma}\label{lem:rigid_emb_unique}
    Every finite $n$-dimensional partial order with realizers of size $m$ admits a unique rigid embedding into $\mathbf{m}^n$.
\end{lemma}

\begin{proof}
    Suppose that $\mathbf{P}=(P,<,<_1,\ldots,<_n)$ is as described and define $f:P\to\{1,\ldots,m\}^n$ by $f(p)=(\mathrm{rk}_1(p),\ldots,\mathrm{rk}_n(p))$ for $p\in P$, where $\mathrm{rk}_i(p)$ is the rank of $p$ with respect to the linear order $<_i$. It is easy to check that $f$ is a rigid embedding.

    For uniqueness, suppose that $g:P\to\{1,\ldots,m\}^n$ is a rigid embedding. Since the image of $g$ contains no colinear points, the lexicographic orders $<_1,\ldots,<_n$ on $\{1,\ldots,m\}^n$ are simply the coordinate orders on the image of $g$. As $g$ must send each of the realizers $<_1,\ldots,<_n$ on $P$ to the coordinate orders on its image, it follows that $g$ must preserve the rank of elements of $p$ with respect to each $<_i$. From this, it follows that $g=f$.
\end{proof}

We can now prove our main theorem:

\begin{proof}[Proof of Theorem \ref{thm:PO_n_RP}]
    Fix $k\in\mathbb{N}$. Suppose that $\mathbf{A}=(A,<,<_1,\ldots,<_n)$ and $\mathbf{B}=(B,<,<_1,\ldots,<_n)$ are in $\mathcal{PO}_{n,<_1,\ldots,<_n}$, and $\mathbf{A}\preceq\mathbf{B}$. Let $l=|A|$ and $m=|B|$, and take $r$ to be as provided by Theorem \ref{thm:prod_Ramsey}. We claim that the structure $\mathbf{C}=\mathbf{r}^n_{<_1,\ldots,<_n}$ witnesses the Ramsey property for $\mathbf{A}$ and $\mathbf{B}$.

    Suppose that $c:\binom{\mathbf{C}}{\mathbf{A}}\to \{1,\ldots,k\}$ is a $k$-colouring. Define a $k$-colouring $\hat{c}$ of the $\mathbf{l}^n$-subgrids of $\mathbf{r}^n$ (without realizers) as follows: If $\mathbf{L}$ is an $\mathbf{l}^n$-subgrid of $\mathbf{r}^n$, let $\hat{c}(\mathbf{L})=c(\mathbf{A_L})$, where $\mathbf{A_L}$ is the unique, by Lemma \ref{lem:rigid_emb_unique}, rigidly embedded copy of $\mathbf{A}$ in $\mathbf{L}$. By our choice of $r$, there is an $\mathbf{m}^n$-subgrid $\mathbf{M}$ of $\mathbf{r}^n$, all of whose $\mathbf{l}^n$-subgrids $\mathbf{L}$ are monochromatic with respect to $\hat{c}$, say $\hat{c}(\mathbf{L})=i$.

    Let $\mathbf{B}_0$ be the rigidly embedded copy of $\mathbf{B}$ in $\mathbf{M}$. Then, if $\mathbf{A}_0$ is a copy of $\mathbf{A}$ in $\mathbf{B}_0$, $\mathbf{A}_0$ will be rigidly embedded into some $\mathbf{l}^n$-subgrid $\mathbf{L}_0$ of $\mathbf{M}$, and so $c(\mathbf{A}_0)=\hat{c}(\mathbf{L}_0)=i$. Thus, $c$ is constant on $\binom{\mathbf{B}_0}{\mathbf{A}}$.
\end{proof} 

Viewing $\mathcal{PO}_{n,<_1,\ldots,<_n}$ as a Fra\"iss\'e order class, Theorems \ref{thm:KPT} and \ref{thm:PO_n_RP} now imply the following, as promised:

\begin{corollary}\label{cor:EA}
    $\mathrm{Aut}(\mathbf{D}_{n,<_1,\ldots,<_n})$ is extremely amenable.\qed
\end{corollary}

\section{The universal minimal flow of \texorpdfstring{$\mathrm{Aut}(\mathbf{D}_n)$}{Aut(Dn,<)}}\label{sec:uni_min_flow}

An action of a Polish group $G$ on a compact (Hausdorff) space is \emph{minimal} if it admits no proper closed invariant subsets, or equivalently, every orbit is dense. In particular, if $G$ is extremely amenable, every minimal flow of $G$ is just the trivial action on a point. It is a general fact (see \cite{MR2140630}) that for any $G$, there is a ``largest'' minimal flow, that is, a minimal action of $G$ on a compact space $X$ such that for any other minimal action of $G$ on a compact space $Y$, there is a continuous surjection $\phi:X\to Y$ which is also a \emph{$G$-homomorphism}, meaning:
\[
    \phi(g\cdot x)=g\cdot \phi(x)
\]
for all $g\in G$ and $x\in X$. Moreover, this space $X$, together with the action of $G$, is unique in the sense that for any other compact space $X'$ and minimal action of $G$ on $X'$ with the above property, there is a homeomorphism $X\to X'$ which is also a $G$-homomorphism. The action of $G$ on $X$ is called the \emph{universal minimal flow} of $G$.

In this section, we will determine the universal minimal flow of $\mathrm{Aut}(\mathbf{D}_n)$. First, we note the following:

\begin{proposition}
    $\mathrm{Aut}(\mathbf{D}_n)$ is not extremely amenable.
\end{proposition}

\begin{proof}
    Let $X$ be the set of all linear orders on $D_n$ which extend $<$. We can view $X$ as a closed, and thus compact, subset of the product $\{0,1\}^{D_n^2}$. Then, $\mathrm{Aut}(\mathbf{D}_n)$ acts continuously on $X$ via the \emph{logic action}, $g\,\cdot\!\prec\;=\;\prec^g$, where:
    \[
        \mathbf{a}\prec^g\mathbf{b} \iff g^{-1}(\mathbf{a})\prec g^{-1}(\mathbf{b}),
    \]
    for all $\mathbf{a},\mathbf{b}\in D_n$, $\prec\;\in X$, and $g\in\mathrm{Aut}(\mathbf{D}_n)$. 
    
    Take $\prec\;\in X$ and consider $\mathbf{a},\mathbf{b}\in D_n$ which are incomparable with respect to $<$, but such that $\mathbf{a}\prec\mathbf{b}$. Since the map which exchanges $\mathbf{a}$ and $\mathbf{b}$ is an automorphism of $(\{\mathbf{a},\mathbf{b}\},<)$, by Corollary \ref{cor:D_n_2-hom} there is an $h\in\mathrm{Aut}(\mathbf{D}_n)$ such that $h$ exchanges $\mathbf{a}$ and $\mathbf{b}$. But then, $\mathbf{b}\prec^{h^{-1}}\mathbf{a}$, so $\prec^{h^{-1}}\;\neq\;\prec$. Thus, the action of $\mathrm{Aut}(\mathbf{D}_n)$ on $X$ has no fixed points.
\end{proof}

The main tool for determining the universal minimal flow of automorphism groups of reducts of linearly ordered structures is Theorem 7.5 in \cite{MR2140630}. However, we cannot directly apply this result in the case of $\mathbf{D}_n=(D_n,<)$, viewed as a reduct of $\mathbf{D}_{n,<_1,\ldots,<_n}=(D_n,<,<_1,\ldots,<_n)$, because the latter is an expansion by more than just a single linear order. Even if we could adapt the aforementioned theorem to this setting, any generalization of the ``reasonable'' condition in its hypotheses would likely fail in our case, as $\mathbf{D}_n$ is not a Fra\"iss\'e limit (cf.~Proposition 5.2 in \cite{MR2140630}).

Instead, we will show directly that the universal minimal flow of $\mathrm{Aut}(\mathbf{D}_n)$ is its extended logic action on the set of all realizers for $\mathbf{D}_n$, considered as $n$-tuples of linear orders on $D_n$. The key observation is that this set consists merely of permutations of the coordinate orders on $D_n$. A weak version of this fact for products of finite linear orders was proved on p.~797 of \cite{MR0641492}. As usual, we denote by $S_n$ the group of permutations of $\{1,\ldots,n\}$.

\begin{lemma}[\cite{MR0641492}]\label{lem:permute_realizers_finite}
    Suppose that $(F_1,<_1),\ldots, (F_n,<_n)$ are finite linear orderings with at least two elements and $(F_1\times\cdots\times F_n,<)$ their direct product. If $\prec_1,\ldots,\prec_n$ are linear orderings on $F_1\times\cdots\times F_n$ which realize $<$, then there is a permutation $\sigma\in S_n$ such that
    \[
        a_i<_ib_i \implies \mathbf{a}\prec_{\sigma(i)}\mathbf{b},
    \]
    for all $i\leq n$ and $\mathbf{a},\mathbf{b}\in F_1\times\cdots\times F_n$.\qed
\end{lemma}

This property extends to products of arbitrary linear orders, and in particular, to $(\mathbb{Q}^n,<)$.

\begin{lemma}\label{lem:permute_realizers}
    Suppose that $(L_1,<_1),\ldots, (L_n,<_n)$ are linear orderings with at least two elements and $(L_1\times\cdots\times L_n,<)$ their direct product. If $\prec_1,\ldots,\prec_n$ are linear orderings on $L_1\times\cdots\times L_n$ which realize $<$, then there is a permutation $\sigma\in S_n$ such that
    \[
        a_i<_ib_i \implies \mathbf{a}\prec_{\sigma(i)}\mathbf{b},
    \]
    for all $i\leq n$ and $\mathbf{a},\mathbf{b}\in L_1\times\cdots\times L_n$.
\end{lemma}

\begin{proof}
    Enumerate $S_n=\{\sigma_1,\ldots,\sigma_{n!}\}$ and suppose the result fails. Then, for each $j\leq n!$, there are $\mathbf{a}^j=(a_1^j,\ldots,a_n^j)$ and $\mathbf{b}^j=(b_1^j,\ldots,b_n^j)$ in $L_1\times\cdots\times L_n$, and an $i_j\leq n$, such that 
    \[
        a_{i_j}^j<_{i_j} b_{i_j}^j \quad\text{but}\quad \mathbf{a}^j\not\prec_{\sigma(i_j)}\mathbf{b}^j.
    \]
    For each $i\leq n$, let $F_i\subseteq L_i$ be the projections of $\{\mathbf{a}^j,\mathbf{b}^j:j\leq n!\}$ onto the $i$th coordinate, possibly with one additional point from $L_i$ to ensure $|F_i|\geq 2$. Then, letting $<_i$ be the order on $F_i$ inherited from $L_i$, $(F_1,<_1),\ldots,(F_n,<_n)$ contradict Lemma \ref{lem:permute_realizers_finite}.
\end{proof}

To apply the preceding lemma to $\mathbf{D}_n$, we must be able to extend realizers for $\mathbf{D}_n$ to realizers for $(\mathbb{Q}^n,<)$.

\begin{lemma}\label{lem:extend_realizers}
     If $\prec_1,\ldots,\prec_n$ are linear orderings on $D_n$ which realize the product order $<$, then there are linear extensions $\prec_1^*,\ldots,\prec_n^*$ of $\prec_1,\ldots,\prec_n$, respectively, to $\mathbb{Q}^n$ which realize the product order there.
\end{lemma}

\begin{proof}
    For each $i\leq n$, let $\prec_i'$ be the union of $\prec_i$ on $D_n$ and the product order $<$ on $\mathbb{Q}^n$, and $\prec_i''$ its transitive closure. We claim that $\prec_i''$ is irreflexive, and thus a partial order on $\mathbb{Q}^n$. To see this, we will show by induction that there are no $\prec_i'$-cycles of length $m\geq 1$:
    \[
        \mathbf{a}^1\prec_i'\mathbf{a}^2\prec_i'\cdots\prec_i'\mathbf{a}^m\prec_i'\mathbf{a}^1
    \]
    in $\mathbb{Q}^n$. When $m=1$, this is immediate from the fact that $\prec_i$ and $<$ are irreflexive. When $m=2$, there are four possible forms of $\prec_i'$-cycles:
    \begin{align*}
        &\mathbf{a}^1<\mathbf{a}^2<\mathbf{a}^1\\
        &\mathbf{a}^1<\mathbf{a}^2\prec_i\mathbf{a}^1\\
        &\mathbf{a}^1\prec_i\mathbf{a}^2<\mathbf{a}^1\\
        &\mathbf{a}^1\prec_i\mathbf{a}^2\prec_i\mathbf{a}^1.
    \end{align*}
    The first case implies $\mathbf{a}^1<\mathbf{a}^1$, contrary to $<$ being irreflexive. In the remaining three cases, both $\mathbf{a}^1$ and $\mathbf{a}^2$ must lie in $D_n$, and $\prec_i$ extends $<$ there, so we get $\mathbf{a}^1\prec_i\mathbf{a}^1$, contrary to $\prec_i$ being irreflexive.

    Assume the claim holds for $k\geq 1$ and that we are given a $\prec_i'$-cycle of length $k+1$:
    \[
        \mathbf{a}^1\prec_i'\mathbf{a}^2\prec_i'\cdots\prec_i'\mathbf{a}^k\prec_i'\mathbf{a}^{k+1}\prec_i'\mathbf{a}^1.
    \]
    If anywhere in this cycle, two successive instances of $\prec_i'$ are given by the same relation, either both $<$ or both $\prec_i$, then transitivity of that relation allows us to reduce to a cycle of length $k$, to which we can apply the induction hypothesis and obtain a contradiction. Thus, we may assume the relations $<$ and $\prec_i$ alternate in this cycle. If $k+1$ is odd, then the first and last relation are the same, but then the cycle we get by starting at $\mathbf{a}^{k+1}$:
    \[
        \mathbf{a}^{k+1}\prec_i'\mathbf{a}^1\prec_i'\mathbf{a}^2\cdots\prec_i'\mathbf{a}^k\prec_i'\mathbf{a}^{k+1}
    \]
    begins with two successive instances of the same relation, as in the prior case. Thus, the only remaining possibilities are that $k+1$ is even and either:
    \[
        \mathbf{a}^1<\mathbf{a}^2\prec_i\cdots\prec_i\mathbf{a}^k<\mathbf{a}^{k+1}\prec_i\mathbf{a}^1
    \]
    or
    \[
        \mathbf{a}^1\prec_i\mathbf{a}^2<\cdots<\mathbf{a}^k\prec_i\mathbf{a}^{k+1}<\mathbf{a}^1.
    \]
    In the former case, it must be that $\mathbf{a}^k$, $\mathbf{a}^{k+1}$, and $\mathbf{a}^1$ are all in $D_n$, and since $\prec_i$ extends $<$ on $D_n$, we have $\mathbf{a}^k\prec_i\mathbf{a}^{k+1}\prec_i\mathbf{a}^1$. Then, we can reduce the cycle by replacing the last two inequalities with $\mathbf{a}^k\prec_i\mathbf{a}^1$ and apply the induction hypothesis. In the other case, we can start the cycle at $\mathbf{a}^2$ instead and argue similarly.

    Thus, for each $i\leq n$, $\prec_i''$ is a partial order on $\mathbb{Q}^n$ extending both $\prec_i$ on $D_n$ and $<$ on $\mathbb{Q}^n$. Apply Lemma \ref{lem:szpilrajn} to obtain a linear order $\prec_i^*$ on $\mathbb{Q}^n$ which extends $\prec_i''$, and let $<^*\;=\;\prec_1^*\cap\cdots\cap\prec_n^*$. Clearly, $<\;\subseteq\;<^*$, and $<^*$ agrees with $<$ on $D_n$. We claim that $<\;=\;<^*$ everywhere.

     Suppose that $<^*\;\not\subseteq\;<$. Then, there are $\mathbf{a},\mathbf{b}\in\mathbb{Q}^n$ such that $\mathbf{a}<^*\mathbf{b}$, but $\mathbf{a}\not<\mathbf{b}$. Since $<\;\subseteq\;<^*$, we cannot have $\mathbf{b}<\mathbf{a}$, so $\mathbf{a}$ and $\mathbf{b}$ are incomparable with respect to the product order. That is, there are $i\neq j\leq n$ such that $a_i<b_i$, but $b_j<a_j$.

     Choose $x,y\in\mathbb{Q}$ such that $a_i<x<b_i$ and $b_j<y<a_j$. Next, choose $\mathbf{a}'=(a_1',\ldots,a_n')$ and $\mathbf{b}'=(b_1',\ldots,b_n')$ in $D_n$ such that:
     \begin{align*}
         \mathbf{a}'&<\mathbf{a} & \mathbf{b}&<\mathbf{b}'\\
         %a_i'&<x & x&<b_i'\\
         y&<a_j' & b_j'&<y,
     \end{align*}
     which is possible by the density of $D_n$ in $\mathbb{Q}^n$. Note that $\mathbf{a}'$ and $\mathbf{b}'$ are incomparable with respect to the product order. Then,
     \[
         \mathbf{a}'<\mathbf{a}<^*\mathbf{b}<\mathbf{b}',
     \]
     so by transitivity of $<^*$ and the fact that it extends $<$, we have that $\mathbf{a}'<^*\mathbf{b}'$. But, $\mathbf{a}'$ and $\mathbf{b}'$ are both in $D_n$ and $<^*$ agrees with $<$ on $D_n$, so $\mathbf{a}'<\mathbf{b}'$, contradicting that they were incomparable. Hence, $<\;=\;<^*$.
\end{proof}

Applying Lemmas \ref{lem:extend_realizers} and \ref{lem:permute_realizers}, we have the following, where we gain the reverse implication $\Leftarrow$ from the fact that $D_n$ contains no colinear points:

\begin{lemma}\label{lem:permute_realizers_D_n}
    If $\prec_1,\ldots,\prec_n$ are linear orderings on $D_n$ which realize the product order $<$, then there is a permutation $\sigma\in S_n$ such that
    \[
        a_i<b_i \iff \mathbf{a}\prec_{\sigma(i)}\mathbf{b},
    \]
    for all $i\leq n$ and $\mathbf{a},\mathbf{b}\in D_n$.\qed
\end{lemma}

Let $\mathcal{R}_n$ be the set of all realizers on $\mathbf{D}_n$, viewed as $n$-tuples $(\prec_1,\ldots,\prec_n)$ of linear orders on $D_n$ which realize the product order $<$. By Lemma \ref{lem:permute_realizers_D_n}, $\mathcal{R}_n$ is exactly the set of permutations of the $n$ coordinate orders on $D_n$. In particular, $|\mathcal{R}_n|=n!$.

\begin{lemma}\label{lem:logic_action_transitive}
    $\mathrm{Aut}(\mathbf{D}_n)$ acts continuously and transitively on $\mathcal{R}_n$ via the logic action: $g\cdot(\prec_1,\ldots,\prec_n)=(\prec_1^g,\ldots,\prec_n^g)$, where
    \[
        \mathbf{a}\prec_i^g\mathbf{b} \quad\iff\quad g^{-1}(\mathbf{a})\prec_i g^{-1}(\mathbf{b}),
    \]
    for all $i\leq n$, $\mathbf{a},\mathbf{b}\in D_n$, and $g\in\mathrm{Aut}(\mathbf{D}_n)$.
\end{lemma}

\begin{proof}
    That $\mathrm{Aut}(\mathbf{D}_n)$ acts continuously via the logic action on $n$-tuples of linear orders on $D_n$ is standard. We must verify, however, that it preserves realizers. Suppose $(\prec_1,\ldots,\prec_n)\in\mathcal{R}_n$ and $g\in\mathrm{Aut}(\mathbf{D}_n)$. Since $g$ preserves $<$, we have that $<\;\subseteq\;\prec_1^g\cap\cdots\cap\prec_n^g$. Suppose that $\mathbf{a}\prec_i^g\mathbf{b}$ for all $i\leq n$. Then, $g^{-1}(\mathbf{a})\prec_i g^{-1}(\mathbf{b})$ for all $i\leq n$, so $g^{-1}(\mathbf{a})< g^{-1}(\mathbf{b})$, but then, $\mathbf{a}<\mathbf{b}$, since $g^{-1}\in\mathrm{Aut}(\mathbf{D}_n)$. Thus, $<\;=\;\prec_1^g\cap\cdots\cap\prec_n^g$, and so $(\prec_1^g,\ldots,\prec_n^g)\in\mathcal{R}_n$.

    To see that this action is transitive, take $(\prec_1,\ldots,\prec_n)$ in $\mathcal{R}_n$. Since the truth of the axioms $\mathsf{DPO}_{n,<_1,\ldots,<_n}$ given in Section \ref{sec:axioms} is clearly invariant under permutations of the coordinate orders, it follows by Theorem \ref{thm:axiomatization} that $(D_n,<,\prec_1,\ldots,\prec_n)$ is isomorphic to $\mathbf{D}_{n,<_1,\ldots,<_n}$. This isomorphism is realized by some $g\in\mathrm{Aut}(\mathbf{D}_n)$ such that $g\cdot(<_1,\ldots,<_n)=(\prec_1,\ldots,\prec_n)$. Thus, the orbit of $(<_1,\ldots,<_n)$ is all of $\mathcal{R}_n$.
\end{proof}

\begin{theorem}\label{thm:UMF}
    The logic action of $\mathrm{Aut}(\mathbf{D}_n)$ on $\mathcal{R}_n$ is the universal minimal flow of $\mathrm{Aut}(\mathbf{D}_n)$.
\end{theorem}

\begin{proof}
    For brevity, let us write $G=\mathrm{Aut}(\mathbf{D}_n)$ and $H=\mathrm{Aut}(\mathbf{D}_{n,<_1,\ldots,<_n})$, and note that $H$ is a closed subgroup of $G$. By Lemma \ref{lem:logic_action_transitive}, the action of $G$ on $\mathcal{R}_n$ is a minimal flow. It remains to prove that it is universal. 
    
    Suppose that $X$ is a compact space on which $G$ acts minimally. By Corollary \ref{cor:EA}, $H$ is extremely amenable, so there is some $x_0\in X$ such that $h\cdot x_0=x_0$ for all $h\in H$. We will construct a (trivially continuous) $G$-homomorphism $\phi$ from $\mathcal{R}_n$ onto $X$ such that $\phi(<_1,\ldots,<_n)=x_0$.

    Consider the set
    \[
        \Phi=\{(g\cdot(<_1,\ldots,<_n),g\cdot x_0)\in\mathcal{R}_n\times X:g\in G\}.
    \]
    We claim that $\Phi$ is the graph of the desired function $\phi:\mathcal{R}_n\to X$. Clearly, by the transitivity of the action of $G$ on $\mathcal{R}_n$, the domain of $\Phi$ is all of $\mathcal{R}_n$. To see that $\Phi$ is a function, suppose that $g\cdot(<_1,\ldots,<_n)=h\cdot (<_1,\ldots,<_n)$, for some $g,h\in G$. Then, $h^{-1}g\in H$, so
    \[
        h\cdot x_0=h\cdot (h^{-1}g\cdot x_0)=g\cdot x_0.
    \]
    Thus, $\phi$ is a well-defined $G$-homomorphism, $\phi(<_1,\ldots,<_n)=x_0$, and since $G$ acts minimally on $X$, $\phi$ must map $\mathcal{R}_n$ onto $X$.
\end{proof}

The sets $D_n$ can be constructed to not only be dense and without colinear points, but also invariant under coordinate permutations. To do this, we can modify the proof of Proposition \ref{prop:D_n_exists} so that when choosing points $\mathbf{d}_m$ we avoid the additional finitely many hyperplanes of points which have more than one coordinate in common with themselves or any of the previously chosen points, and then add all coordinate permutations of $\mathbf{d}_m$ to the set being constructed. By Theorem \ref{thm:axiomatization}, the resulting $D_n'$ gives a structure $(D_n',<,<_1,\ldots,<_n)$ isomorphic to $\mathbf{D}_{n,<_1,\ldots,<_n}$, and in particular, $(D_n',<)$ is isomorphic to $\mathbf{D}_n$. Then, $S_n$ acts faithfully on $(D_n',<)$ by permuting coordinates, so we may identify $S_n$ with a subgroup of $G=\mathrm{Aut}(D_n',<)$. Through this identification, $S_n$ normalizes the subgroup $H=\mathrm{Aut}(D_n',<,<_1,\ldots,<_n)$ and $H\cap S_n=\{e\}$. Given $g\in G$, by Lemma \ref{lem:permute_realizers_D_n}, there is some $\sigma_g\in S_n$ such that 
\[
    (D_n',<,<_1^g,\ldots,<_n^g)=(D_n',<,<_{\sigma_g(1)},\ldots,<_{\sigma_g(n)}),
\]
thus $\sigma_g^{-1}g=h_g$ for some $h_g\in H$, so $g=\sigma_gh_g$, and this decomposition is unique. It follows that $G=H\rtimes S_n$. Lemma 5 in \cite{MR2948747} then implies that the action $g\cdot\sigma=\sigma_g\sigma$ of $G$ on $S_n$ is the universal minimal flow of $G$. Clearly, $S_n$ can be identified with $\mathcal{R}_n$ and this action will coincide with the logic action described in Theorem \ref{thm:UMF}, yielding an alternate proof of this result. We will follow a similar approach in determining the universal minimal flow of $\mathrm{Aut}(\mathbb{Q}^n,<)$ in Theorem \ref{thm:UMF_Qn} below.

\section{The rational grid, weak realizers, and full products}\label{sec:Qn}

In this final section, we will analyze the structure of the $n$-dimensional rational grid $\mathbf{Q}_n=(\mathbb{Q}^n,\leq)$ and a certain relevant expansion in a manner similar to what we have done for $\mathbf{D}_n$ and $\mathbf{D}_{n,<_1,\ldots,<_n}$. This analysis will be simpler in many ways, as it will take advantage of the explicit product structure of $\mathbf{Q}_n$. Note that we have moved from the strict product order $<$ on $\mathbb{Q}_n$ to the corresponding non-strict order $\leq$, and will do so for all partial orders in what follows.

To expand the structure $\mathbf{Q}_n$, and the corresponding class of finite substructures, we will utilize the notion of a \emph{weak realizer}, in direct analogy to what we have done for $\mathbf{D}_n$ and $\mathbf{D}_{n,<_1,\ldots,<_n}$. We say that a binary relation $\preceq$ on $L$, and the structure $(L,\preceq)$ as a whole, is a \emph{weak linear order} if it is reflexive, transitive, and total. Given this, the symmetrization of $\preceq$ is equivalence relation $\sim$ on $L$, where $a\sim b$ if and only if $a\preceq b$ and $b\preceq a$ for all $a,b\in L$, and $\preceq$ induces a linear ordering on the quotient $L/\!\sim$.

Given a partial order $(P,\leq)$ and weak linear orders $\leq_1,\ldots,\leq_n$ on $P$, we say that $\leq_1,\ldots,\leq_n$ \emph{realize} $\leq$ if $\leq \,=\,\leq_1\cap\cdots\cap\leq_n$, as before. We call the $n$-tuple $(\leq_1,\ldots,\leq_n)$ a \emph{weak realizer} for $\leq$. Notice that weak linear orders $\leq_1,\ldots,\leq_n$ on a set $P$ realize a partial order if and only if their intersection is antisymmetric, i.e., whenever $a\neq b$ in $P$, there is some $i\leq n$ such that neither $a\leq_i b$ nor $b\leq_i a$.

The most relevant examples of weak linear orders for us are the coordinate orders on a product of linear orders, which in turn form a weak realizer for the product order. We write $\mathbf{Q}_{n,\leq_1,\ldots,\leq_n}=(\mathbb{Q}^n,\leq,\leq_1,\ldots,\leq_n)$ for the expansion of $\mathbf{Q}_n$ by the coordinate orders $\leq_1,\ldots,\leq_n$.

Since every $n$-dimensional partial order embeds into an $n$-fold product of linear orders, every such partial order is realized by $n$ weak linear orders, namely those induced by the coordinate orders on the product. Conversely, if $(P,\leq)$ is realized by weak linear orders $\leq_1,\ldots,\leq_n$, and we let $\widetilde{P}_i=P/\!=_i$, where $=_i$ is the symmetrization of $\leq_i$, have the induced linear order, then it is easy to see that the natural map of $P$ into the product $\widetilde{P}_1\times\cdots\times\widetilde{P}_n$, formed by first using the diagonal embedding into $P^n$ and then taking the quotient map in each coordinate, is an embedding; the only thing to check is that this map is injective, but this follows from the fact that $\leq\,=\,\leq_1\cap\cdots\cap\leq_n$ is antisymmetric. In particular, the minimum size of a weak realizer for a partial order coincides with its dimension.

We call a structure $(Q,\leq,\leq_1,\ldots,\leq_n)$, where $(\leq_1,\ldots,\leq_n)$ is a weak realizer for the partial order $\leq$, an \emph{$n$-dimensional partial order with weak realizers}, and denote by $\mathcal{PO}_{n,\leq_1,\ldots,\leq_n}$ the corresponding \emph{class of all finite $n$-dimensional partial orders with weak realizers}.

The structures described above may be viewed as special cases of the full product construction, in the sense of \cite{MR3210656}, \cite{MR3497266}, and \cite{MR4723481}. In general, given structures $\mathbf{A}$ and $\mathbf{B}$, in possibly different relational signatures, their \emph{full product} is the structure $\mathbf{A}\boxtimes\mathbf{B}$ with domain $A\times B$, and whose signature consists of the union of the signatures of $\mathbf{A}$ and $\mathbf{B}$, together with two new binary relations $=_1$ and $=_2$, intepreted as follows: For all $(a,b),(c,d)\in A\times B$,
\[
    (a,b)R(c,d) \iff aRc,
\]
when $R$ is in the signature of $\mathbf{A}$,
\[
    (a,b)S(c,d) \iff bSd,
\]
when $S$ is in the signature of $\mathbf{B}$, and
\begin{align*}
    (a,b)&=_1(c,d) \iff a=c\\
    (a,b)&=_2(c,d) \iff b=d.
\end{align*}
Likewise, given classes of finite relational structures $\mathcal{K}_1$ and $\mathcal{K}_2$, their \emph{full product} is the class $\mathcal{K}_1\boxtimes\mathcal{K}_2$ of all finite structures which embed into $\mathbf{A}\boxtimes\mathbf{B}$, for some $\mathbf{A}\in\mathcal{K}_1$ and $\mathbf{B}_2\in\mathcal{K}_2$.

In the case of linear orders $(L,\leq)$ and $(K,\leq)$, their resulting full product is the structure $(L\times K,\leq_1,\leq_2,=_1,=_2)$, endowed with the coordinate orders $\leq_1$ and $\leq_2$, and coordinate identities $=_1$ and $=_2$. Iterating this operation, we may view the full product of linear orders $(L_1,\leq)$, ..., $(L_n,\leq)$ as 
\[
    (L_1\times\cdots\times L_n,\leq_1,\ldots,\leq_n,=_1,\ldots,=_n).
\]
Since the product order $\leq$ on $L_1\times\cdots\times L_n$ is definable from $\leq_1,\ldots,\leq_n$, we may add it to the signature, and since the coordinate identities $=_1,\ldots,=_n$ are definable from $\leq_1,\ldots,\leq_n$, we may omit them. So, we may view this structure as an $n$-dimensional partial order with weak realizers. Conversely, any $n$-dimensional partial order with weak realizers may be embedded into an $n$-fold full product of linear orders, as above. Thus, we may identify $\mathbf{Q}_{n,\leq_1,\ldots,\leq_n}$ with the full product of $(\mathbb{Q},\leq)$ with itself $n$ times, and the class $\mathcal{PO}_{n,\leq_1,\ldots,\leq_n}$ with the full product of the class $\mathcal{LO}$ of finite linear orders with itself $n$ times.

We can now apply the following structural results concerning full products to $\mathbf{Q}_{n,\leq_1,\ldots,\leq_n}$ and $\mathcal{PO}_{n,\leq_1,\ldots,\leq_n}$.

\begin{lemma}
    Suppose $\mathbf{A}$ and $\mathbf{B}$ are relational structures.
    \begin{enumerate}
        \item (Lemma 3.4(2) in \cite{MR4723481}) $\mathrm{Age}(\mathbf{A}\boxtimes\mathbf{B})=\mathrm{Age}(\mathbf{A})\boxtimes\mathrm{Age}(\mathbf{B})$.
        \item (Proposition 3.5(2) in \cite{MR4723481}) If $\mathrm{Age}(\mathbf{A})$ and $\mathrm{Age}(\mathbf{B})$ are Fra\"iss\'e classes with limits $\mathbf{A}$ and $\mathbf{B}$, respectively, then $\mathrm{Age}(\mathbf{A})\boxtimes\mathrm{Age}(\mathbf{B})$ is a Fra\"iss\'e class with limit $\mathbf{A}\boxtimes\mathbf{B}$. In particular, $\mathbf{A}\boxtimes\mathbf{B}$ is ultrahomogeneous.
        \item (Proposition 3.1 in \cite{MR3210656}) $\mathrm{Aut}(\mathbf{A}\boxtimes\mathbf{B})\cong\mathrm{Aut}(\mathbf{A})\times\mathrm{Aut}(\mathbf{B})$.
    \end{enumerate}
\end{lemma}

Consequently, we have:

\begin{theorem}\label{thm:Qn_Fraisse}
    $\mathbf{Q}_{n,\leq_1,\ldots,\leq_n}$ is ultrahomogeneous and its age is $\mathcal{PO}_{n,\leq_1,\ldots,\leq_n}$. Consequently, $\mathcal{PO}_{n,\leq_1,\ldots,\leq_n}$ is a Fra\"iss\'e class with limit $\mathbf{Q}_{n,\leq_1,\ldots,\leq_n}$.\qed
\end{theorem}

Note that if $(Q,\leq,\leq_1,\ldots,\leq_n)$ is a finite $n$-dimensional partial order with weak realizers, then (genuine) linear orders on $Q$ can be defined in terms of $\leq_1,\ldots,\leq_n$ similar to the lexicographic orders on a product of linear orders. In particular, we may view $\mathcal{PO}_{n,\leq_1,\ldots,\leq_n}$ as a Fra\"iss\'e order class.

\begin{theorem}
    $\mathrm{Aut}(\mathbf{Q}_{n,\leq_1,\ldots,\leq_n})\cong\mathrm{Aut}(\mathbb{Q},\leq)^n$.  Consequently, $\mathrm{Aut}(\mathbf{Q}_{n,\leq_1,\ldots,\leq_n})$ is extremely amenable and $\mathcal{PO}_{n,\leq_1,\ldots,\leq_n}$ has the Ramsey property.%\qed
\end{theorem}

\begin{proof}
     By Lemma 6.7 in \cite{MR2140630}, a product of extremely amenable groups is extremely amenable, so $\mathrm{Aut}(\mathbf{Q}_{n,\leq_1,\ldots,\leq_n})$ is extremely amenable, and thus by the reverse direction of the Kechris--Pestov--Todor\v{c}evi\'{c} correspondence (Theorem \ref{thm:KPT}), its age $\mathcal{PO}_{n,\leq_1,\ldots,\leq_n}$ has the Ramsey property.
\end{proof}

Note that, as was the case for $\mathcal{PO}_{n,<_1,\ldots,<_n}$, one can also show directly that $\mathcal{PO}_{n,\leq_1,\ldots,\leq_n}$ satisfies the Ramsey property using the Product Ramsey Theorem.

In \cite{MR0641492}, Manaster and Remmel gave a finite axiomatization of $\mathbf{Q}_n$, and we will give a finite axiomatization of $\mathbf{Q}_{n,\leq_1,\ldots,\leq_n}$ here. Much like the axioms $\mathsf{DLO}_{n,<_1,\ldots,<_n}$ above for $\mathbf{D}_{n,<_1,\ldots,<_n}$, the idea is to say that $\leq_1,\ldots,\leq_n$ are weak linear orders which realize the product order, and that every non-trivial region determined by finitely many points contains an additional point. The main difference between $\mathbf{D}_{n,<_1,\ldots,<_n}$ and $\mathbf{Q}_{n,\leq_1,\ldots,\leq_n}$ is that, in addition to open hypercubes, points in $\mathbf{Q}_{n,\leq_1,\ldots,\leq_n}$ may also determine open hyperplane segments and singletons, which are the open sides or corners (really, subcomplexes) of those open hypercubes. 

Given an $n$-dimensional partial order with weak realizers $(Q,\leq,\leq_1,\ldots,\leq_n)$, for each $i\leq n$, let $=_i$ be the symmetrization of $\leq_i$ and $<_i$ the corresponding strict order, both of which are clearly definable. A point $q\in Q$ \emph{determines} $3n$ \emph{regions} in $Q$, via the formulas (with free variable $u$):
\[
    u<_i q, \quad u=_iq, \quad\text{and}\quad q<_iu
\]
for $i\leq n$. More generally, finitely many points in $Q$ \emph{determine} finitely many \emph{regions} as intersections of the regions defined by each point. We call such regions \emph{non-trivial} if the corresponding region in $\mathbb{Q}^n$ is non-empty and, in the case of singletons, distinct from the finitely many points.

\begin{example}
Consider the five points $\mathbf{x}$, $\mathbf{y}$, $\mathbf{z}$, $\mathbf{w}$, and $\mathbf{v}$ lying in $\mathbb{Q}^2$ below. These five points determine $30$ open rectangles, $25$ open horizontal line segments, $24$ open vertical line segments, and $15$ singletons (distinct from $\mathbf{x}$, $\mathbf{y}$, $\mathbf{z}$, $\mathbf{w}$, and $\mathbf{v}$).

For example, the shaded open rectangle is described by the formula (in free variable $u$)
\[
    \mathbf{z}<_1u \wedge u<_1 \mathbf{w} \wedge \mathbf{w} <_2 u \wedge u<_2 \mathbf{x}.
\]
The open line segment, with left endpoint at $\mathbf{y}$, is described by
\[
    \mathbf{y} <_1 u \wedge u<_1 \mathbf{z} \wedge u=_2 \mathbf{y}.
\]
And the singleton, indicated by the grey node, is described by
\[
    u=_1 \mathbf{y} \wedge u=_2 \mathbf{x}.
\]
\begin{figure}[ht]
\scalebox{0.5}{
\begin{tikzpicture}
    \draw[very thick][<->] (0,8) -- (0,-2);
    \draw[very thick][<->] (8,0) -- (-2,0);
    
    \tikzstyle{node_style}=[shape=circle,minimum size=0.01cm,very thin, draw=black, fill=black]
    \tikzstyle{node_outline}=[shape=circle,minimum size=0.5cm,very thin, draw=black]
    \tikzstyle{nodegray}=[shape=circle,minimum size=0.01cm,very thin, draw=gray, fill=gray]
    \tikzstyle{nodecross}=[shape=cross,minimum size=0.01cm,very thin, draw=black, fill=dark gray]
    
    \newcounter{p}
    \newcounter{q}
    \foreach \p in {3,1,4,6}{
    \draw[thick, lightgray] (\p,8) -- (\p,-2);
    }
    \foreach \q in {1,2,3,5,6}{
    \draw[thick, lightgray] (8,\q) -- (-2,\q);
    }
    %\draw[very thick,dotted] (6,8) -- (6,-3);
    %\draw[very thick,dotted] (8,3) -- (-2,3);
    %\pgfsetfillopacity{0.2}
    \draw [very thick,dotted, fill=lightgray] (4,3) rectangle (6,5);
    %\draw [draw=white,very thick,fill=gray] (3.2,1.9) rectangle (3.9,2.1);
    \draw [line width=2 pt,gray] (i1) -- (4,2); 

    \node[node_style] (i1) at (3,2) {};
    \node[below left=3 pt of {(3,2)}, outer sep=3pt] {\LARGE$\mathbf{y}$}; 
    \node[node_style] (i2) at (1,5) {};
    \node[below left=3 pt of {(1,5)}, outer sep=3pt] {\LARGE$\mathbf{x}$}; 
    \node[node_style] (i3) at (4,6) {};
    \node[below left=3 pt of {(4,6)}, outer sep=3pt] {\LARGE$\mathbf{z}$}; 
    \node[node_style] (r) at (6,3) {};
    \node[below left=3 pt of {(6,3)}, outer sep=3pt] {\LARGE$\mathbf{w}$}; 
    \node[node_style] (s) at (6,1) {};
    \node[below left=3 pt of {(6,1)}, outer sep=3pt] {\LARGE$\mathbf{v}$};
    \node[nodegray] (u) at (3,5) {};
\end{tikzpicture}}
\caption{Five points and the regions they determine in $\mathbf{Q}_{2,\leq_1,\leq_2}$.}
\end{figure}
\end{example}

The following axioms for a structure of the form $(Q,\leq,\leq_1,\ldots,\leq_n)$ say that $\leq$ is a partial order, $\leq_1,\ldots,\leq_n$ are weak linear orders which realize $\leq$, and every non-trivial region determined by finitely many points in $Q$ contains a point. Since such regions in $\mathbb{Q}^n$ are (possibly unbounded) $n$-dimensional open hypercubes and their open sides or corners, they are again uniquely determined by at most $2^n$ points.

\begin{definition}
    The theory $\mathsf{DLO}_{n,\leq_1,\ldots,\leq_n}$ of \emph{dense $n$-dimensional lattice orders with weak realizers} consists of the following sentences:
    \begin{itemize}
        \item $\leq$ is a (non-strict) partial ordering:
        \begin{align*}
            &\forall x(x\leq x)\\
            &\forall x\forall y(x\leq y\wedge y\leq x \rightarrow x=y)\\
            &\forall x\forall y\forall z(x\leq y\wedge y\leq z \rightarrow x\leq z).
        \end{align*}
        \item Each of $\leq_1,\ldots,\leq_n$ is a weak linear ordering:
        \begin{align*}
            &\forall x(x\leq_i x)\\
            %&\forall x\forall y(x<_iy\rightarrow y\not<_i x)\\
            &\forall x\forall y\forall z(x\leq_iy\wedge y\leq_i z \rightarrow x\leq_i z)\\
            &\forall x\forall y(x\leq_i y \vee y\leq_i x)
        \end{align*}
        for $i\leq n$.
        \item $\leq_1,\ldots,\leq_n$ realize $\leq$:
        \begin{align*}
            \forall x\forall y(x\leq y \leftrightarrow (x\leq_1y\wedge\cdots\wedge x\leq_n y)).
        \end{align*}
        \item $\leq_1,\ldots,\leq_n$ are dense and unbounded:
        \begin{align*}
            \forall x_1\cdots\forall x_{2^n} 
            \left[\bigwedge_{i=1}^{\mathrm{N}(n)}\left(\theta_i(x_1,\ldots,x_{2^n})\rightarrow
            \bigwedge_{j=1}^{s(i)}\rho^i_j(x_1,\ldots,x_{2^n})\right)\right],
        \end{align*}
        where:
        \begin{itemize}
            \item  $\mathrm{N}(n)$ is the number of possible configurations of $\leq 2^n$ many points with respect to $<_1,\ldots,<_n$ and $=_1,\ldots,=_n$, 
            \item for each $i\leq\mathrm{N}(n)$, $\theta_i(x_1,\ldots,x_{2^n})$ asserts that $x_1,\ldots,x_{2^n}$ are in the $i$th such configuration (for configurations of fewer than $2^n$ points, we may assert that some of the $x_k$ are equal),
            \item $s(i)$ is the number of non-trivial regions determined by $x_1,\ldots,x_{2^n}$ in the $i$th configuration, and
            \item for each $j\leq s(i)$, $\rho^i_j(x_1,\ldots,x_{2^n})$ asserts the existence of a point the $j$th non-trivial region determined by $x_1,\ldots,x_{2^n}$. 
        \end{itemize}
    \end{itemize}
\end{definition}

We note that any such structure $(Q,\leq,\leq_1,\ldots,\leq_n)$ is necessarily a lattice: Given $x,y\in Q$, since each $\leq_i$ is total, we can define the minimum $\min_i\{x,y\}$ and maximum $\max_i\{x,y\}$ with respect to $\leq_i$, where if $x=_iy$, we may choose either $x$ or $y$. Then, if we let $u$ witness the formula
\[
    \exists u(u=_1\min{\!}_1\{x,y\} \wedge \cdots \wedge u=_n\min{\!}_n\{x,y\}),
\]
and $v$ witness the formula
\[
    \exists v(v=_1\max{\!}_1\{x,y\} \wedge \cdots \wedge u=_n\max{\!}_n\{x,y\}),
\]
it follows that $u$ is the meet, and $v$ the join, of $x$ and $y$ with respect to $\leq$. 

Clearly, $\mathbf{Q}_{n,\leq_1,\ldots,\leq_n}$ satisfies $\mathsf{DLO}_{n,\leq_1,\ldots,\leq_n}$. To see that $\mathbf{Q}_{n,\leq_1,\ldots,\leq_n}$ is the only countable model of these axioms up to isomorphism, one may proceed by a back-and-forth argument exactly as in the proof of Theorem \ref{thm:axiomatization} whose details we leave to the reader.

\begin{theorem}
    $\mathbf{Q}_{n,\leq_1,\ldots,\leq_n}$ is the unique countable dense $n$-dimensional lattice order with weak realizers, up to isomorphism.\qed
\end{theorem}

Finally, we consider automorphisms of the reduct $\mathbf{Q}_n=(\mathbb{Q}^n,\leq)$. Recall that, by (the proof of) Proposition \ref{prop:D_n_not_hom} and Corollary \ref{cor:D_n_2-hom}, $\mathbf{D}_n$ is $2$-homogeneous but not $3$-homogeneous. %We claim that $\mathbf{Q}_n$ has one fewer degree of homogeneity.

\begin{proposition}
    $\mathbf{Q}_n$ is $1$-homogeneous but not $2$-homogeneous.
\end{proposition}

\begin{proof}
    That $\mathbf{Q}_n$ is $1$-homogeneous, namely that any two points can be mapped to each other by an automorphism, follows immediately from Theorem \ref{thm:Qn_Fraisse}. To see that it is not $2$-homogeneous, consider points $\mathbf{a}$, $\mathbf{b}$, $\mathbf{x}$, $\mathbf{y}$ in $\mathbb{Q}^2$ such that 
    \[
        a_1 <x_1<y_1< b_1 \text{ and } a_i <y_i<x_i< b_i,
    \]
    for $2\leq i\leq n$, so then $\mathbf{a} \leq \mathbf{x} \leq \mathbf{b}$, $\mathbf{a} \leq \mathbf{y} \leq \mathbf{b}$, and $\mathbf{x},\mathbf{y}$ are incomparable.

    \begin{figure}[ht]
    \scalebox{.75}{
    \centering
    \begin{tikzpicture}
        \tikzstyle{node_style}=[shape=circle,minimum size=0.01cm,very thin, draw=black, fill=black]
        \tikzstyle{red}=[shape=circle,minimum size=0.01cm,very thin, draw=gray, fill=gray]
        \draw[->] (0,0) -- (0,5);
        \draw[->] (0,0) -- (5,0);
        \draw[->] (8,0) -- (13,0);
        \draw[->] (8,0) -- (8,5);

        \node[node_style] (a) at (1,1) {};
        \node[left=3 pt of {(1,1)}, outer sep=3pt,fill=white] {$\mathbf{a}$};        
        \node[node_style] (b) at (4,4) {};
        \node[left=3 pt of {(4,4)}, outer sep=3pt,fill=white] {$\mathbf{b}$};
        \node[red] (x) at (2,3) {};
        \node[left=3 pt of {(2,3)}, outer sep=3pt,fill=white] {$\mathbf{x}$}; 
        \node[red] (y) at (3,2) {};
        \node[left=3 pt of {(3,2)}, outer sep=3pt,fill=white] {$\mathbf{y}$}; 

        \node[node_style] (d) at (11,1) {};
        \node[right=3 pt of {(11,1)}, outer sep=3pt,fill=white] {$\mathbf{a}'$};
        \node[node_style] (e) at (11,4) {};
        \node[right=3 pt of {(11,4)}, outer sep=3pt,fill=white] {$\mathbf{b}'$};
        \node[red] (f) at (11,3) {};
        \node[right=3 pt of {(11,3)}, outer sep=3pt,fill=white] {$\mathbf{x}'$};
        \node[red] (g) at (11,2) {};
        \node[right=3 pt of {(11,2)}, outer sep=3pt,fill=white] {$\mathbf{y}'$};

        \draw[thin,->] (a) -- (d);
        \draw[thin,->] (b) -- (e);
        \draw[thin,dotted,->] (x) -- (f);
        \draw[thin,dotted,->] (y) -- (g);
    \end{tikzpicture}}
    \caption{$\mathbf{Q}_2$ is not $2$-homogeneous.}
    \end{figure}
    
    Take points $\mathbf{a}'$, $\mathbf{b}'$ in $\mathbb{Q}^2$ such that 
    \[
        a'_1 = b'_1 \text{ and } a'_i < b'_i
    \]
    for $2\leq i\leq n$. The mapping $f$ given by $f(\mathbf{a})=\mathbf{a}'$ and $f(\mathbf{b})=\mathbf{b}'$ is then an isomorphism from $(\{\mathbf{a},\mathbf{b}\},\leq)$ to $(\{\mathbf{a}',\mathbf{b}'\},\leq)$. But, there is no way to extend $f$ to an automorphism $g$ on all of $\mathbb{Q}^n$. To see this, note that if $g$ was such an automorphism, say with $g(\mathbf{x})=\mathbf{x}'$ and $g(\mathbf{y})=\mathbf{y}'$, then we must have that $\mathbf{a}'\leq \mathbf{x}'$ and $\mathbf{y}'\leq\mathbf{b}$, but since $a_1'=b_1'$, this will imply that $a_1'\leq x_1'$ and $y_1'\leq b_1'$, and thus, $\mathbf{x}'$ and $\mathbf{y}'$ must be comparable. This contradicts that $\mathbf{x}$ and $\mathbf{y}$ are incomparable, and $g$ is an embedding.
\end{proof}

In \cite{MR0641492}, Manaster and Remmel defined $\mathbf{b}_1,\ldots,\mathbf{b}_n$ in $\mathbb{Q}^n$ to be a \emph{basis} if they are pairwise incomparable and have pairwise common meet; this property is clearly expressible in terms of the product order, and thus is preserved by any automorphism of $\mathbf{Q}_n$. We will say that a basis $\mathbf{b}_1,\ldots,\mathbf{b}_n$ such that
\begin{align*}
    \mathbf{b}_{1}&=(X_1,x_2,\ldots,x_n)\\
    \mathbf{b}_{2}&=(x_1,X_2,\ldots,x_n)\\
    \vdots&\\
    \mathbf{b}_{n}&=(x_1,x_2,\ldots,X_n),
\end{align*}
where $x_i<X_i$ for all $i\leq n$, is in \emph{standard position}. For example, the standard basis $\mathbf{e}_1=(1,0,\ldots,0), \mathbf{e}_2=(0,1,\ldots,0), \ldots, \mathbf{e}_n=(0,0,\ldots,1)$ is in standard position. The lemma on p.~795 of \cite{MR0641492} shows that any basis is a permutation of one in standard position; we call this permutation the \emph{type} of the basis and view $S_n$ as the set of \emph{basis types}. We also consider $S_n$ as a subgroup of $\mathrm{Aut}(\mathbf{Q}_n)$ via its action permuting coordinates.

Given $g\in\mathrm{Aut}(\mathbf{Q}_n)$, let $\sigma_g$ be the type of the image of the standard basis under $g$, that is, the unique element of $S_n$ such that $\sigma_g^{-1}(g(\mathbf{e}_1)),\ldots,\sigma_g^{-1}(g(\mathbf{e}_n))$ is in standard position.

It is shown on pp.~795--796 of \cite{MR0641492} that for each $i\leq n$, there is a formula in the language of partial orders which, given \emph{any} basis of $\mathbf{Q}_n$ in standard position, defines the $i$th coordinate order $\leq_i$. Consequently, any element of $\mathrm{Aut}(\mathbf{Q}_n)$ which maps the standard basis to another basis in standard position must preserve coordinates, i.e., is an element of $\mathrm{Aut}(\mathbf{Q}_{n,\leq_1,\ldots,\leq_n})$. In particular, for any $g\in\mathrm{Aut}(\mathbf{Q}_n)$, $\sigma_g^{-1}g\in\mathrm{Aut}(\mathbf{Q}_{n,\leq_1,\ldots,\leq_n})$. Note that elements of $S_n$ clearly commute with those of $\mathrm{Aut}(\mathbf{Q}_{n,\leq_1,\ldots,\leq_n})$.

\begin{lemma}
    The mapping $\mathrm{Aut}(\mathbf{Q}_n)\to S_n$ given by $g\mapsto \sigma_g$ is a continuous homomorphism with kernel $\mathrm{Aut}(\mathbf{Q}_{n,\leq_1,\ldots,\leq_n})$.
\end{lemma}

\begin{proof}
    Let $g,h\in\mathrm{Aut}(\mathbf{Q}_n)$. Then, for each $i\leq n$,
    \begin{align*}
        (\sigma_g\sigma_h)^{-1}(gh(\mathbf{e}_i))&=\sigma_h^{-1}\sigma_g^{-1}gh(\mathbf{e}_i)\\
        &=\sigma_g^{-1}g(\sigma_h^{-1}h(\mathbf{e}_i))
    \end{align*}
    since $\sigma_h^{-1}$ commutes with $\sigma_g^{-1}g$. As both $\sigma_h^{-1}h$ and $\sigma_g^{-1}g$ preserve coordinates, $\sigma_g^{-1}g(\sigma_h^{-1}h(\mathbf{e}_1)),\ldots,\sigma_g^{-1}g(\sigma_h^{-1}h(\mathbf{e}_n))$ must be in standard position. Thus, $\sigma_g\sigma_h=\sigma_{gh}$. Clearly, the elements in the kernel of this mapping are exactly those which map the standard basis to a basis in standard position, thus preserve coordinates. To see that this mapping is continuous, note that whether $\sigma_g=\sigma\in S_n$ is determined by the image of finitely many points in $\mathbb{Q}^n$ under $g$, and thus there is a clopen set of such $g\in\mathrm{Aut}(\mathbf{Q}_n)$.
\end{proof}

We can now describe the universal minimal flow of $\mathrm{Aut}(\mathbf{Q}_n)$.

\begin{theorem}\label{thm:UMF_Qn}
    $\mathrm{Aut}(\mathbf{Q}_n)=\mathrm{Aut}(\mathbf{Q}_{n,\leq_1,\ldots,\leq_n})\rtimes S_n$. Consequently, the universal minimal flow of $\mathrm{Aut}(\mathbf{Q}_n)$ is its natural action on basis types in $S_n$.
\end{theorem}

\begin{proof}
    We have already pointed out that $\mathrm{Aut}(\mathbf{Q}_{n,\leq_1,\ldots,\leq_n})$ is the kernel of a homomorphism, thus is normal in $\mathrm{Aut}(\mathbf{Q}_n)$, and clearly has trivial intersection with $S_n$. For any $g\in\mathrm{Aut}(\mathbf{Q}_n)$, $g\sigma_g^{-1}=(\sigma_gg^{-1})^{-1}=h\in\mathrm{Aut}(\mathbf{Q}_{n,\leq_1,\ldots,\leq_n})$, so $g=h\sigma_g$. Hence, $\mathrm{Aut}(\mathbf{Q}_n)=\mathrm{Aut}(\mathbf{Q}_{n,\leq_1,\ldots,\leq_n})\rtimes S_n$. The remainder follows by Lemma 5 in \cite{MR2948747}.
\end{proof}

Finally, note that, in analogy to our results for $\mathrm{Aut}(\mathbf{D}_n)$, we could have characterized the universal minimal flow of $\mathrm{Aut}(\mathbf{Q}_n)$ as its action on weak realizers, which are nothing more than permutations of the coordinate orders, using the fact that they can be defined from a basis. This amounts to the same thing as the action on basis types, which seems like a more natural characterization given the results in \cite{MR0641492}.

%\newpage

\bibliographystyle{abbrv}
\bibliography{posets.bib}

\end{document}